\documentclass[11pt,a4paper]{amsart}
\usepackage{graphicx}
\usepackage{amsmath,amsfonts}
\usepackage{comment}
\usepackage{amsthm,amssymb,latexsym}
\usepackage{url}
\usepackage{hyperref}
\usepackage[active]{srcltx}
\usepackage[usenames]{color}
\usepackage[utf8]{inputenc}
\usepackage{enumitem}
\usepackage[normalem]{ulem}
\setlist[enumerate]{
    leftmargin=*,
    labelsep=0pt,
    align=left
}
\usepackage{xcolor}

\newcommand{\bfs}{\boldsymbol}
\newtheorem{theorem}{Theorem}[section]

\newtheorem{lemma}[theorem]{Lemma}

\newtheorem*{fact*}{Fact}

\theoremstyle{definition}

\newtheorem*{claim}{Claim}
\newtheorem{remark}[theorem]{Remark}
\numberwithin{equation}{section}

\newcommand{\A}{\mathbb A}
\newcommand{\F}{\mathbb F}
\newcommand{\K}{\mathbb K}

\newcommand{\Pp}{\mathbb P}

\newcommand{\fq}{\F_{\hskip-0.7mm q}}

\newcommand{\cfq}{\overline{\F}_{\hskip-0.7mm q}}
\def\ifm#1#2{\relax \ifmmode#1\else#2\fi}

\newcommand{\klk}    {\ifm {,\ldots,} {$,\ldots,$}}

\begin{document}

\title{Power Sum Varieties over Finite Fields}

\author[N. Giménez]{Nardo Giménez${}^{1,2}$}

\author[J. F.  Gottig]{Juan Francisco Gottig${}^{1,2}$}

\author[M. Palacios]{Miguel Palacios${}^{1,2}$}

\author[M. Pérez]{Mariana Pérez${}^{1,2}$}

\author[M.  Privitelli]{Melina Privitelli${}^{1,2}$}

\address{${}^{1}$Universidad Nacional de Hurlingham, Laboratorio de Investigación en Geometría y Álgebra (LIGA) \\ Av. Gdor. Vergara 2222 (B1688GEZ), Villa Tesei, Buenos Aires, Argentina}
\email{\nolinkurl{{nardo.gimenez,francisco.gottig,miguelangel.palacios,
mariana.perez,melina.privitelli}@unahur.edu.ar}}
\address{${}^{2}$ Consejo Nacional de Investigaciones Científicas y Técnicas (CONICET),
Ar\-gentina}

\thanks{The authors were partially supported by the grants
PIUNAHUR$2025$ $PIU2025TG_{}{20}$}

\keywords{Finite fields, power sums, diagonal equations, $\fq$--rational solutions, estimates, Waring's problem}
\maketitle

\section*{Abstract}
We study a family of affine algebraic varieties over a finite field $\mathbb{F}_q$
defined by systems of power sum polynomials. Our main goal is to investigate their
geometric and arithmetic properties, with particular emphasis on the structure and
dimension of their singular loci. Using tools from the theory of symmetric polynomials
and generalized Vandermonde determinants, we establish explicit conditions under which
these varieties are complete intersections and prove that, in a broad range of
parameters, their singular loci have codimension at least two. As a consequence, we
derive normality and irreducibility results for both the affine varieties and their
projective closures. Building on this geometric analysis, we obtain estimates for the
number of $\mathbb{F}_q$--rational points and deduce effective existence results for
solutions of systems of symmetric power sum equations over finite fields.
Finally, we apply  these results to the study of Waring's problem and to the analysis of the cost of computing a $\mathbb{F}_q$--rational point of such systems.

\section{Introduction}

 Let $\mathbb{F}_q$ denote the finite field with $q$ elements and $\mathbb{F}_q[X_1,\ldots,X_n]$ the polynomial ring in the indeterminates $X_1,\ldots,X_n$ over $\mathbb{F}_q$.  For a given positive integer $d$ the \emph{(symmetric)  power sum polynomial} $p_{d}^{(n)}$ of order $d$ in $\mathbb{F}_q[X_1,\ldots,X_n]$ is defined by $p_{d}^{(n)}:=\sum_{i=1}^{n} X_i^{d}$. In this paper we study polynomial systems defined by power sum polynomials with coefficients in a finite field. More precisely, we consider systems of the form
\begin{equation}\label{intro: power sum eq}
  p_{d_1}^{(n)} = b_1, \ldots ,p_{d_m}^{(n)}  =b_m,
\end{equation}
with positive orders $d_1 < d_2 < \cdots < d_m$ and $b_1, \ldots, b_m$ any elements of  $\mathbb{F}_q$.

Systems defined by power sum polynomials have attracted considerable
attention in recent years, mainly in connection with problems on regular sequences, symmetric
polynomial ideals, and complete intersections. For systems defined over
algebraically closed fields of characteristic zero, significant progress has
been achieved by Conca et al.~\cite{CSS2024}, who established sharp conditions
under which families of power sum polynomials define complete intersections,
normal domains, and unique factorization domains.
However, in  positive characteristic, systems defined by power
sum polynomials has been less studied. Although several works address special cases like systems consisting of one or two equations,
or low-dimensional settings, results on geometric and arithmetic properties of systems involving
three or more power sum equations over finite fields remain limited.
For instance,
in \cite{Cao2016} the authors study systems of power sum equations in the square
case, when the number of equations equals the number of variables. From an
arithmetic perspective, explicit existence results of $\mathbb{F}_q$-rational points and estimates of their number are known only in a few specific situations,
the most classical being Fermat varieties, corresponding to diagonal
hypersurfaces, which have been extensively studied using character sum
techniques (see, for example, \cite{Weil1949, IrelandRosen1990}).

The first objective of this paper is to extend some of the known results for power sum systems
established in characteristic zero to the finite field setting.

Our approach relies on tools from the theory of symmetric polynomials and generalized
Vandermonde determinants in combination with techniques from algebraic geometry.

Our first main result (see, Theorems \ref{teo: V es interseccion completa}, \ref{teo:props_clausura_infinito} and \ref{teo: pcl V interseccion completa}) provides explicit conditions under which collections of power sum polynomials form regular sequences and define complete
intersections.

\begin{theorem}\label{thm:intro-geom} Let $m, n$ be integers with $2\le m\le n$. Let $p_{d_1}^{(n)},\ldots,p_{d_m}^{(n)}$ be the symmetric power sum polynomials of positive orders $d_1 < \cdots < d_m$ and $b_1 \klk b_m$ any elements in $\fq$. Set $G:=d_m-d_1-m+1$. Assume that the characteristic $p$ of $\mathbb{F}_q$ does not divide any $d_i$ for $1\leq i \leq m$. Then:
\begin{enumerate}
\item[i)] If $n \ge 2m+G-1$, then $p_{d_1}^{(n)}-b_1,\ldots,p_{d_m}^{(n)}-b_m$ is a regular sequence in
$\mathbb{F}_q[X_1,\ldots,X_n]$.
\item[ii)] If $n \ge 2m+G$, then the ideal $(p_{d_1}^{(n)}-b_1,\ldots,p_{d_m}^{(n)}-b_m)$ is radical.
\item[iii)] {If further $p>n-m-G+2$, then for $m\geq 3$ with  $n \ge 2m+G$ the polynomials $p_{d_1}^{(n)}-b_1X_0^{d_1},\ldots,p_{d_m}^{(n)}-b_mX_0^{d_m}$ define a normal complete intersection variety of degree $d_1\cdots d_m$ in $\mathbb{P}^{n}$. Moreover, $p_{d_1}^{(n)},\ldots,p_{d_m}^{(n)}$ define a normal complete intersection variety in $\mathbb{P}^{n-1}$.}
\end{enumerate}
\end{theorem}

As a consequence, under suitable conditions on the number of variables and the
degrees involved, power sum varieties over finite fields share structural
features analogous to those observed in characteristic zero, such as being
complete intersections whose singular loci have codimension at least two.

Furthermore, this geometric and algebraic information is also relevant from a computational perspective. Properties of a polynomial system, such as defining a complete intersection, generating a radical ideal and the absolute irreducibility of the underlying variety, provide suitable conditions for efficiently solving polynomial systems and searching for rational points using available algorithms (see, e.g, [1]).

The same geometric framework also underlies our arithmetic analysis of power-sum varieties. In particular, we derive explicit estimates for the number of $\mathbb{F}_q$--rational points, which lead to effective existence results for solutions over finite fields. See Theorem \ref{prop: existencia} for details.

\begin{theorem}\label{thm:intro-existence}
With hypotheses as in Theorem \ref{thm:intro-geom},
if the number $n$ of variables is sufficiently large with respect to a fixed \(m\) and fixed degrees \(d_1,\ldots,d_m\), then the existence of an \(\mathbb F_q\)--rational point is guaranteed whenever
\(
q>(3+md_m)^2.
\)
\end{theorem}

These results complement and extend our previous work  \cite{GottigPerezPrivitelli2024}. In that paper, an estimate for the number of
rational solutions of certain systems of diagonal equations with
general exponents was developed, with applications to the moment subset sum
problem. The present work strengthens and extends that analysis in
three directions: we allow arbitrary, possibly nonconsecutive
exponents through the parameter $G$; we give a proof, based
on symmetric polynomials and generalized Vandermonde determinants,
that the associated varieties are normal complete intersections
under explicit hypotheses; and we improve the exponent of $q$ in the
point-counting estimate from $q^{n/2}$ to
$q^{n/2 - (m-1-2G)/4}$, yielding a genuine power saving whenever
$2G < m-1$.

More recently, Podestá and Videla \cite{PV2026} considered monic diagonal systems with arbitrary numbers of equations and variables and distinct exponents. Their approach is fundamentally different from ours: they translate the counting problem into graph-theoretic and character-sum problems, expressing the number of solutions in terms of walks, adjacency matrices, and spectra of associated diagonal generalized Paley graphs, as well as in terms of Weil sums. For general systems, these identities provide exact representations of the number of solutions, but not closed explicit formulas, since their evaluation still requires the computation of spectral data of the associated graphs or of the corresponding Weil sums, which may itself be difficult. Explicit counting formulas are obtained only for special families, in concrete, they study monic homogeneous Hermitian-form diagonal systems. Their Weil-sum estimate, moreover, requires the non-degeneracy of every polynomial \(f_\alpha(T)=\alpha_1T^{d_1}+\cdots+\alpha_mT^{d_m}\), a condition that may be nontrivial to verify in practice. Thus, while their work is primarily arithmetic and graph-theoretic, it does not address the algebraic-geometric structure of the varieties defined by these systems, which is the main focus of this paper.

The paper is organized as follows. In Section~2 we recall the basic notions and results
from algebraic geometry and symmetric polynomials that will be used throughout the
paper. In Section~3 we study the geometric properties of power sum varieties over finite
fields, including bounds on the dimension of their singular loci, the complete intersection
property, and the geometry of their projective closures. In Section~4 we combine these
geometric results with point-counting techniques to derive existence results for
$\mathbb{F}_q$--rational solutions of systems of power sum equations. In Section 5 we study the application to the Generalized
Waring's problem over finite fields. Finally, in Section 6 we discuss the cost of computing a rational solution of system \eqref{intro: power sum eq} by means of the algorithm of \cite{CaMa06b}.

\section{Basic notions and notations}

In this section we collect the basic definitions and facts of
algebraic geometry we need for our results. We use standard
notions and notations which can be found in, e.g., \cite{Kunz85} or
\cite{Shafarevich94}.

Let $\K$ be any of the fields $\fq$ or $\cfq$, the closure of $\fq$. We denote by $\A^n$
the affine $n$--dimensional space $\cfq{\!}^{n}$ and by $\Pp^n$ the
projective $n$--dimensional space over $\cfq$.

{A subset $V\subset \A^n$ is an {\em affine variety of
	$\A^n$ defined over} $\K$ (or an affine $\K$--variety for short) if it is the set of
common zeros in $\A^n$ of polynomials $F_1,\ldots, F_{m}$ in $\K[X_1,\ldots, X_n]$.
A subset $V\subset \Pp^n$ is a {\em projective variety defined over}
$\K$ (or a projective $\K$--variety for short) if it is the set of
common zeros in $\Pp^n$ of homogeneous polynomials $F_1,\ldots, F_m
$ in $\K[X_0 ,\ldots, X_n]$. In both cases we say that $V$ is \emph{defined} by the polynomials $F_1,\ldots, F_m$. We shall denote by
$\{F_1=0,\ldots, F_m=0\}$ or $V(F_1,\ldots,F_m)$ the affine or
projective $\K$--variety consisting of the common zeros of
$F_1,\ldots, F_m$.  Both spaces $\A^n$ and  $\Pp^n$  are endowed with their respective Zariski topologies over $\K$, for
which the closed sets are the affine or the projective $\K$--varieties respectively.  We think a projective or affine $\K$--variety
to be endowed with the induced Zariski topology.}

In the remaining part of this section, unless otherwise stated, all
results referring to varieties in general should be understood as
valid for both projective and affine varieties.

A $\K$--variety $V$ is $\K$--{\em irreducible} (or {\em irreducible} for short if it cannot be expressed
as a finite union of proper $\K$--subvarieties of $V$. Further, $V$
is {\em absolutely irreducible} if it is $\cfq$--irreducible as a
$\cfq$--variety. Any $\K$--variety $V$ can be expressed as an
irredundant union $V=\mathcal{C}_1\cup \cdots\cup\mathcal{C}_s$ of
irreducible $\K$--varieties, unique up to
reordering, called the {\em irreducible} $\K$--{\em components} of $V$.

For a $\K$--variety $V$ contained in $\Pp^n$ or $\A^n$, its {\em
	defining ideal} $I(V)$ is the set of polynomials of $\K[X_0,\ldots,
X_n]$, or of $\K[X_1,\ldots, X_n]$, vanishing on $V$. The {\em
	coordinate ring} $\K[V]$ of $V$ is the quotient ring
$\K[X_0,\ldots,X_n]/I(V)$ or $\K[X_1,\ldots,X_n]/I(V)$. The {\em
	dimension} $\dim V$ of $V$ is the length $r$ of a longest chain
$V_0\varsubsetneq V_1 \varsubsetneq\cdots \varsubsetneq V_{r}$ of
nonempty irreducible $\K$--varieties contained in $V$.
We say that $V$ has {\em pure dimension} $r$ if every irreducible
$\K$--component of $V$ has dimension $r$ furthermore, if $W$ is a subvariety of $V$, then the number $\dim V - \dim W$ is
called the {\it{codimension}} of $W$ in $V$.
A $\K$--variety of $\Pp^n$
or $\A^n$ of pure dimension $n-1$ is called a $\K$--{\em
	hypersurface}. A $\K$--hypersurface of $\Pp^n$ (or $\A^n$) can also
be described as the set of zeros of a single nonzero polynomial of
$\K[X_0,\ldots, X_n]$ (or of $\K[X_1,\ldots, X_n]$).

Consider the embedding of $\A^n$ into the projective space $\Pp^n$
which assigns to any $\bfs{x}:=(x_1,\ldots, x_n)\in\A^n$ the point
$(1:x_1:\dots:x_n)\in\Pp^n$. Then, if $V$ is an affine variety in $\A^n$, the projective closure
$\mathrm{pcl}(V)\subset\Pp^n$ of the image of $V$ under this
embedding in the Zariski topology of $\Pp^n$ is called the
projective closure of $V$. {It holds that the irreducible components of $\mathrm{pcl}(V)$ are precisely the closures of the irreducible components of $V$  and that $V$ and $\mathrm{pcl}(V)$ have the same dimension (see, e.g., \cite[Propositions
I.5.17 and II.4.1]{Kunz85}). The points of $\mathrm{pcl}(V)$ lying
in the hyperplane $\{X_0=0\}$ are called the points of
$V$ at infinity. The projective variety
\[
V^{\infty} := \mathrm{pcl}(V) \cap \{X_0 = 0\}  \subset \mathbb{P}^{n-1}
\]
 is called \emph{the locus of $V$ at infinity}. The \emph{affine cone} $\widetilde{W}$ of a projective variety $W\subset \Pp^n$ is the set of all $(x_0 \klk x_n)\in \A^{n+1}$ that occur as homogeneous coordinate systems of a point of $W$, together with the point $(0 \klk 0)\in \A^{n+1}$. $\widetilde{W}$ is an affine variety of $\A^{n+1}$ with $\dim \widetilde{W}=\dim W+1$ (see, e.g., \cite[Proposition II.4.4]{Kunz85}).}

Let $V\subset\A^n$ be a $\K$--variety, $I(V)\subset \K[X_1,\ldots,
X_n]$ its defining ideal and $\boldsymbol{x}$ a point of $V$. The {\em dimension}
$\dim_{\boldsymbol{x}}V$ {\em of} $V$ {\em at} $\boldsymbol{x}$ is the maximum of the dimensions
of the irreducible $\K$--components of $V$ containing $\boldsymbol{x}$. Let $\mathrm{J}(F_1\klk F_m)(\boldsymbol{x})$ denote the Jacobian matrix of $F_1,\ldots, F_m$
with respect to $X_1,\ldots, X_n$ at $\boldsymbol{x}$, namely the matrix $(\partial
F_i/\partial X_j)(\boldsymbol{x})_{1\le i\le m,1\le j\le n}$. If
$I(V)=(F_1,\ldots, F_m)$, the {\em tangent space} ${T}_xV$
to $V$ at $\boldsymbol{x}$ is the kernel of the Jacobian matrix $\mathrm{J}(F_1\klk F_m)(\boldsymbol{x})$. We have
$\dim{T}_{\boldsymbol{x}}V\ge \dim_{\boldsymbol{x}}V$ (see, e.g., \cite[page
94]{Shafarevich94}). The point $\boldsymbol{x}$ is {\em regular} if
$\dim{T}_{\boldsymbol{x}}V=\dim_{\boldsymbol{x}}V$; otherwise, $\boldsymbol{x}$ is called {\em
	singular}. The set of singular points of $V$, denoted $\mathrm{Sing}(V)$, is the {\em singular
	locus} of $V$; it is a closed $\K$--subvariety of
$V$. A variety is called {\em nonsingular} if its singular locus is
empty. For projective varieties, the concepts of tangent space,
regular and singular point can be defined by considering an affine
neighborhood of the point under consideration.

{The {\em degree} $\deg V$ of an irreducible $\K$--variety $V$ is the
maximum of $|V\cap L|$, considering all the linear spaces $L$ of
codimension $\dim V$ such that $|V\cap L|<\infty$. More generally,
following \cite{Heintz83} (see also \cite{Fulton84}), if
$V=\mathcal{C}_1\cup\cdots\cup \mathcal{C}_s$ is the decomposition
of $V$ into irreducible $\K$--components, we define the degree of
$V$ as
$$\deg V:=\sum_{i=1}^s\deg \mathcal{C}_i.$$
The degree of a $\K$--hypersurface $V$ is the degree of a polynomial
of minimal degree defining $V$. Furthermore, if $V$ is an affine variety in $\A^n$, its projective closure $\mathrm{pcl}
(V)$ has  degree equal to $\deg V$ (see, e.g.,
\cite[Proposition 1.11]{CaGaHe91}).}

{Finally, we have the following {\em B\'ezout inequality} (see
\cite{Fulton84, Heintz83, Vogel84}): if $V$ and $W$
are $\K$--varieties of the same ambient space, then
\begin{equation}\label{eq: Bezout}
\deg (V\cap W)\le \deg V \cdot \deg W.
\end{equation}

For a projective variety $V\subset\Pp^n$ or an affine variety
$V\subset\A^n$, we denote by $V(\fq)$ the set of $\fq$--rational
points of $V$, namely $V(\fq):=V\cap \Pp^n(\fq)$ in the projective
case and $V(\fq):=V\cap \A^n(\fq)$ in the affine case.

\subsection{Complete intersections}\label{subsec: complete intersections}
Let $V$ be a $d$-dimensional $\K$-variety in the affine space $\A^n$ or the projective space $\Pp^n$. $V$ is called an \emph{ideal–theoretic complete intersection} (complete intersection for short) if its vanishing ideal $I(V)$ is generated by $n-d$ polynomials in $\K[X_1, \ldots, X_n]$ in the affine case or $n-d$ homogeneous polynomials in $\mathbb{K}[X_0,\ldots,X_n]$ in the projective case. On the other hand, $V$ is called a \emph{set-theoretic complete intersection} if $V$ is the intersection of $n-d$ $\K$-hypersurfaces.

A sequence of polynomials $F_1,\ldots, F_m$ in $\mathbb{K}[X_1,\ldots,X_n]$ is called a \emph{regular sequence} if the ideal $(F_1, \ldots, F_m)$ they generate in $\K[X_1, \ldots,X_n]$ is a proper ideal, and for $1 \leq i \leq m$,  $F_i$ is  not a zero divisor in the quotient
ring $\mathbb{K}[X_1,\ldots,X_n]/ (F_1,\ldots,F_{i-1})$. If in addition, $(F_1,\ldots, F_s)$ is a radical ideal for all $1\le s \le m$, then $F_1,\ldots, F_m$ is called a \emph{reduced regular sequence}.

A regular sequence of polynomials $F_1,\ldots, F_m$ in $\mathbb{K}[X_1,\ldots,X_n]$ (or homogeneous polynomials in $\mathbb{K}[X_0, \ldots, X_n]$) define  a set-theoretic complete intersection of $\A^n$ (respectively of $\Pp^n$) of dimension $n-m$.

In what follows we
shall use the following result (see, e.g., \cite[Chapitre 3,
Remarque 2.2]{Lejeune84}).
\begin{lemma}\label{lemma: reg seq and complete int}
If $F_1 \klk F_m$ are  $m\le n$ homogeneous polynomials in $\K[X_1
\klk X_n]$, the following conditions  are equivalent:
   \begin{itemize}
     \item $F_1 \klk F_m$ define a set theoretic complete intersection of $\A^n$.
     \item $F_1 \klk F_m$  is a regular sequence in $\K[X_1 \klk X_n]$.
   \end{itemize}
\end{lemma}
The following Lemma is a particular case of a more general result proved in \cite[Lemma 5.4]{MaPePr2019}.
\begin{lemma} \label{lemma: comp hom de mayor grado suc reg}
Let $F_1,\ldots,F_m \in \mathbb{K}[X_1,\ldots,X_n]$, and denote by
$F_i^{\deg F_i}$ the homogeneous component of $F_i$ of degree $\deg F_i$,
for $1\leq i\leq m$. If
$F_1^{\deg F_1},\ldots,F_m^{\deg F_m}$
form a regular sequence in $\mathbb{K}[X_1,\ldots,X_n]$, then
$F_1,\ldots,F_m$ form a regular sequence in
$\mathbb{K}[X_1,\ldots,X_n]$.
\end{lemma}

 We shall frequently use the following criterion to
prove that a variety is a complete intersection, which is a rephrase in geometric terms of
\cite[Theorem 18.15]{Eisenbud95}.

\begin{theorem}\label{theorem: eisenbud 18.15}
	Let $F_1,\ldots,F_m\in\mathbb{K}[X_1,\ldots,X_n]$ be polynomials
	which form a regular sequence and let
	$V:=V(F_1,\ldots,F_m)\subset\A^n$.  If the subvariety
	of $V$ defined by the set of common zeros of the maximal minors of
	$\mathrm{J}(F_1\klk F_m)$ has codimension at least one in $V$, then $F_1,\ldots,F_m$ define a radical ideal. In particular,
	$V$ is a complete intersection.
\end{theorem}

If $V\subset\Pp^n$ is a complete intersection defined over $\K$ of
dimension $n-m$, and $F_1 ,\ldots, F_m$ is a system of homogeneous
generators of $I(V)$,
of degrees $d_1,\ldots, d_m$ respectively. In this case, a stronger version of
\eqref{eq: Bezout} holds, called the {\em B\'ezout theorem} (see,
e.g., \cite[Section 2.5.6a]{Man2018})
\begin{equation*}\label{eq: Bezout theorem}
\deg V=d_1\cdots d_m.
\end{equation*}

A complete intersection $V$ is called {\em normal} if it is {\em
	regular in codimension 1}, that is, the singular locus
$\mathrm{Sing}(V)$ of $V$ has codimension at least $2$ in $V$,
namely $\dim V-\dim \mathrm{Sing}(V)\ge 2$ (actually, normality is a
general notion that agrees on complete intersections with the one we
define here). A fundamental result for projective complete
intersections is the Hartshorne connectedness theorem (see, e.g.,
\cite[Theorem VI.4.2]{Kunz85}): if $V\subset\Pp^n$ is a complete
intersection defined over $\K$ and $W\subset V$ is any
$\K$--subvariety of codimension at least 2, then $V\setminus W$ is
connected in the Zariski topology of $\Pp^n$ over $\K$. Applying the
Hartshorne connectedness theorem with $W:=\mathrm{Sing}(V)$, one
deduces the following result.

\begin{theorem}\label{theorem: normal complete int implies irred}
	If $V\subset\Pp^n$ is a normal complete intersection, then $V$ is
	absolutely irreducible.
\end{theorem}

\subsection{Vandermonde determinants}

We shall use a result on generalized Vandermonde determinants due to S. De Marchi (see \cite[Theorem 3.2]{Marchi}).
Given the generalized Vandermonde-type determinant as follows:

\begin{equation} \label{eq:det5}
V_{m_{n-1},\ldots,m_1}(X) =
\begin{vmatrix}
1 & X_1^{m_1} & \cdots & X_1^{m_{n-1}} \\
1 & X_2^{m_1} & \cdots & X_2^{m_{n-1}} \\
\vdots & \vdots & & \vdots \\
1 & X_{n-1}^{m_1} & \cdots & X_{n-1}^{m_{n-1}} \\
1 & X^{m_1} & \cdots & X^{m_{n-1}}
\end{vmatrix},
\end{equation}
where $m_1<m_2<\cdots<m_{n-1}$ are positive integers. The following theorem provides an explicit formula for its determinant.
\begin{theorem}[{\cite[Theorem 3.2]{Marchi}}]\label{thm:main}
Let $\alpha = (m_{n-1}, m_{n-2}, \dots, m_1, 0)$ and
$\delta = (n-1, n-2, \dots, 1, 0)$ be the partitions of length $n$ associated with the determinant \eqref{eq:det5}. Let $\lambda = \alpha - \delta$ and let $\lambda'$ be its conjugate partition. Hence,
\begin{equation} \label{eq:det6}
V_{m_{n-1},\ldots,m_1}(X) = \operatorname{VDM}(X_1,\ldots,X_{n-1})
\left( \prod_{i=1}^{n-1}(X - X_i) \right) P_M(X),
\end{equation}
where $\operatorname{VDM}(X_1\klk X_{n-1})$ is the classical Vandermonde determinant, $M = m_{n-1} - n + 1$, and
\begin{equation} \label{eq:det7}
P_M(X) =
\begin{vmatrix}
(-1)^{\lambda'_1 - 1} e_{\lambda'_1 - 1}^{(n-1)} & (-1)^{\lambda'_1} e_{\lambda'_1}^{(n-1)} & \cdots & (-1)^{\lambda'_1 + M - 1} e_{\lambda'_1 + M - 1}^{(n-1)} \\
\vdots & \vdots & & \vdots \\
(-1)^{\lambda'_M - M} e_{\lambda'_M - M}^{(n-1)} & (-1)^{\lambda'_M - M + 1} e_{\lambda'_M - M + 1}^{(n-1)} & \cdots & (-1)^{\lambda'_M} e_{\lambda'_M}^{(n-1)} \\
1 & X & \cdots & X^M
\end{vmatrix}.
\end{equation}
\end{theorem}

\section{Geometric properties of the power sum  varieties}
Let $n,m$ and $d_1, d_2 \klk d_m$  positive integers with $2 \le m \le n$ and
$d_1 < d_2 < \cdots < d_m$.

For technical reasons in the following we consider polynomial systems defined by polynomials of the form
\begin{equation}\label{eq:intro-system}
f_j := a_1 X_1^{d_j} + \cdots +a_n X_n^{d_j} - b_j,
\qquad 1 \le j \le m,
\end{equation}
with coefficients $a_i \in \fq \setminus\{0\}$ for $1 \leq i \leq n$ and $b_j \in \mathbb{F}_q$. We denote by $\mathcal{V}$ the corresponding affine variety  $\mathcal{V}:=V(f_1\klk f_m)\subset \A^n$.
 Observe that in the case $b_j=0$ for $1\leq j \leq m $ and $a_i=1$  for  $1\leq i \leq n$, we have $f_j=p_{d_j}^{(n)}$.

 Denote by $g_j$ the $j$-th difference $d_j - d_{j-1} - 1$, for $2 \leq j \leq m$. Note that  $g_j=0$ if and only if the exponents $d_j$ and $d_{j-1}$ are consecutive. We set $G = g_2 + \cdots + g_m=d_m-d_1-m+1$.

One of the purposes of this section is to obtain an upper bound for the dimension of $\mathrm{Sing}(\mathcal{V})$, the singular locus of $\mathcal{V}$. To do this, we analyze the following set
\begin{equation}\label{Conjunto C}
C=\left\{\boldsymbol{x}\in \mathcal{V}:\mathrm{rank}\, \mathrm{J}(f_1\klk f_m)(\boldsymbol{x}) <m\right\}.
\end{equation}

Since $\mathrm{Sing}(\mathcal{V})\subset C$, we begin by analyzing a key characterization of the points in $C$.

We shall use the following two technical lemmas.

Denote by $e^{(n)}_d$ $d$-th elementary symmetric polynomial of degree $d$ in $n$ variables $X_1, \ldots,X_n $ namely \[
    e^{(n)}_d=e_d(X_1, \dots, X_n) \;=\; \sum_{1 \le i_1 < \cdots < i_d \le n}
    X_{i_1} \cdots X_{i_d},\,\,\text{if } d\geq 1,
    \]
    $e_0^{(n)}=1$ and  $e_d^{(n)}=0$ if $d<0$ or $d> n$.
\begin{lemma}
\label{lowering the degree}
Let $n,m$ be positive integers and let
$y_{1},\ldots,y_{{n+m}}$ in any field $K$,
such that, $y_n\klk y_{n+m}$ are pairwise distinct. If
\[
\begin{vmatrix}
            e^{(n)}_{a_{11}} & \cdots & e^{(n)}_{a_{1m}}\\
            \vdots & & \vdots \\
            e^{(n)}_{a_{m1}} & \cdots & e^{(n)}_{a_{mm}}\\
        \end{vmatrix}(y_1\klk y_{n-1},y_l)=0
\quad
\text{for all } \ell=n,\ldots,n+m.\]
Then,
$$
\begin{vmatrix}
            e^{(n-1)}_{a_{11}-1} & \cdots & e^{(n-1)}_{a_{1m}-1}\\
            \vdots & & \vdots \\
            e^{(n-1)}_{a_{m1}-1} & \cdots & e^{(n-1)}_{a_{mm}-1}\\
        \end{vmatrix}(y_{1},\ldots,y_{n-1})=0.
$$
\end{lemma}

\begin{proof}
We will use the identity
\[
e^{(n)}_i = e^{(n-1)}_i + X_{n}\, e^{(n-1)}_{i-1}.
\]

Note that,

\begin{align*}
        \begin{vmatrix}
            e^{(n)}_{a_{11}} & \cdots & e^{(n)}_{a_{1m}}\\
            \vdots & & \vdots \\
            e^{(n)}_{a_{m1}} & \cdots & e^{(n)}_{a_{mm}}\\
        \end{vmatrix}
        &=
        \begin{vmatrix}
            e^{(n-1)}_{a_{11}} + X_{n} e^{(n-1)}_{a_{11}-1}
            & \cdots &
            e^{(n-1)}_{a_{1m}} + X_{n} e^{(n-1)}_{a_{1m}-1}\\
            \vdots & & \vdots \\
            e^{(n-1)}_{a_{m1}} + X_{n} e^{(n-1)}_{a_{m1}-1}
            & \cdots &
            e^{(n-1)}_{a_{mm}} + {X_n} e^{(n-1)}_{a_{mm}-1}\\
        \end{vmatrix} \\
        &=
        \underbrace{
        \begin{vmatrix}
            e^{(n-1)}_{a_{11}}& \cdots & e^{(n-1)}_{a_{1m}}\\
            \vdots & & \vdots \\
            e^{(n-1)}_{a_{m1}} & \cdots & e^{(n-1)}_{a_{mm}}\\
        \end{vmatrix}
        }_{B_0} + B_1 X_{n} + \cdots +
        \underbrace{
        \begin{vmatrix}
            e^{(n-1)}_{a_{11}-1} & \cdots & e^{(n-1)}_{a_{1m}-1}\\
            \vdots & & \vdots \\
            e^{(n-1)}_{a_{m1}-1} & \cdots & e^{(n-1)}_{a_{mm}-1}\\
        \end{vmatrix}
        }_{B_m}
        X_{n}^m \\
        &=
        B_0 + B_1 X_{_n} + \cdots + B_{m-1} X_{_n}^{m-1} + B_m X_{n}^m\\
        &=:D\in K[X_1\klk X_{n-1}][X_n]
\end{align*}

Let $\bfs{y}_{n-1}:=(y_1,\klk, y_{n-1})$.
Observe that $D(\bfs{y}_{n-1},X_n)\in K[X_n]$, a polynomial of degree at most $m$, has at least $m+1$ distinct roots, $y_n\klk y_{n+m}$, by hypothesis. Then $B_m(\bfs{y}_{n-1})=0$ follows.
\end{proof}

\begin{lemma}\label{cada m+M} Assume that the characteristic $p$ of $\fq$ does not divide any $d_j$ for $1\leq j \leq m$.
Then every subsequence of \( \boldsymbol{x} \in C\)   with \( m+G \) nonzero entries must contain at least two equal entries.

\end{lemma}
\begin{proof}
Let $\boldsymbol{x}=(x_1, \ldots, x_n)$.
Toward a contradiction, suppose that there is a subsequence $(x_{\alpha_1}, \ldots, x_{\alpha_{m+G}})$ of $\boldsymbol{x}$ with $m+G$ pairwise distinct nonzero entries.
Let $\boldsymbol{y}=(x_{\alpha_{k_1}}, \ldots, x_{\alpha_{k_m}})$ be any permutation of a subsequence of length $m$ of $(x_{\alpha_1}, \ldots, x_{\alpha_{m+G}})$.
We consider the  $(m\times m)-$minor of $\mathrm{J}(f_1\klk f_m)(\boldsymbol{x})$ involving the variables $X_{\alpha_{k_1}}, \ldots, X_{\alpha_{k_m}}$, which up to a sign is equal to the determinant

\begin{align*}
  \begin{vmatrix}
       a_{\alpha_{k_1}} d_1 x_{\alpha_{k_1}}^{d_1 - 1} & \cdots & a_{\alpha_{k_m}} d_1 x_{\alpha_{k_m}}^{d_1 - 1} \\
        \vdots & \ddots & \vdots \\
       a_{\alpha_{k_1}} d_ m  x_{\alpha_{k_1}}^{d_ m -1 } & \cdots & a_{\alpha_{k_m}} d_ m  x_{\alpha_{k_m}}^{d_ m -1 }
    \end{vmatrix}
    = \bigg(\prod_{i=1}^{m} a_{\alpha_{k_i}}d_ix_{\alpha_{k_i}}^{d_1 - 1}  \bigg)
    \begin{vmatrix}
        1 & \cdots & 1 \\
        x_{\alpha_{k_1}}^{d_2 - d_1} & \cdots & x_{\alpha_{k_m}}^{d_2 - d_1} \\
        \vdots & & \vdots \\
        x_{\alpha_{k_1}}^{d_m - d_1} & \cdots & x_{\alpha_{k_m}}^{d_m - d_1}
    \end{vmatrix}.
\end{align*}
As $\bfs{x}\in C$, this determinant vanishes. Since, by the hypothesis, $\prod_{j=1}^{m} a_{\alpha_{k_j}} d_jx_{\alpha_{k_j}}^{d_1-1}$
is a nonzero element of $\cfq$, it follows that
\begin{align*}
    \begin{vmatrix}
        1 & \cdots & 1 \\
        x_{\alpha_{k_1}}^{d_2 - d_1} & \cdots & x_{\alpha_{k_m}}^{d_2 - d_1} \\
        \vdots & & \vdots \\
        x_{\alpha_{k_1}}^{d_m-d_1} & \cdots & x_{\alpha_{k_m}}^{d_m-d_1}
    \end{vmatrix}
    = 0.
\end{align*}

By Theorem \ref{thm:main}, we have

\begin{align*}
    \begin{vmatrix}
        1 & \cdots & 1 \\
        x_{\alpha_{k_1}}^{d_2 - d_1} & \cdots & x_{\alpha_{k_m}}^{d_2 - d_1} \\
        \vdots & & \vdots \\
        x_{\alpha_{k_1}}^{d_m - d_1} & \cdots & x_{\alpha_{k_m}}^{d_m - d_1}
    \end{vmatrix}
    = \mathrm{VDM}(x_{\alpha_{k_1}}, \ldots, x_{\alpha_{k_m}})  P_G(x_{\alpha_{k_1}}, \ldots, x_{\alpha_{k_{m-1}}}, x_{\alpha_{k_m}}),
\end{align*}

where $P_G$ denotes the polynomial in $\fq[X_1, \ldots, X_{m-1}, X]$ defined in $\eqref{eq:det7}$ and $G=d_m-d_1-m +1$. Since $x_{\alpha_{k_1}}, \ldots, x_{\alpha_{k_m}}$ are pairwise distinct, $\mathrm{VDM}(x_{\alpha_{k_1}}, \ldots, x_{\alpha_{k_m}})\neq 0$ and
we conclude that $P_G(\boldsymbol{y}) = 0$ for any permutation $\boldsymbol{y}=(x_{\alpha_{k_1}}, \ldots, x_{\alpha_{k_m}})$ of a subsequence of length $m$ of $(x_{\alpha_1}, \ldots, x_{\alpha_{m+G}})$.

According to Theorem \ref{thm:main} we define the  vectors:
\begin{align*}
    \alpha&:=(m-1+G, d_{m-1}-d_1\klk d_2-d_1,0),\\
    \delta&:=(m-1,m-2\klk2,1,0),\\
    \lambda&=\alpha-\delta=(\underbrace{G\klk G}_{r_1},\underbrace{G-1\klk G-1}_{r_2}\klk\underbrace{1\klk 1}_{r_{G}}\underbrace{0\klk 0}_{r_{G+1}}),\\
    \lambda'&=\left(\sum_{j=1}^{G}r_j,\sum_{j=1}^{G-1}r_j\klk r_1+r_2,r_1\right):=(\lambda_1'\klk\lambda'_G)\end{align*}
where $\lambda'$ is the conjugate partition of $\lambda$, $\sum_{j=1}^{G+1}r_j=m$ and $r_1, r_{G+1}\neq 0$. Observe that $m\geq \lambda'_1 \geq \lambda'_2\geq \cdots \geq \lambda'_G \geq 1$.

We have
\begin{align}\label{def:PM}
P_G =&\begin{vmatrix}
    (-1)^{\lambda'_1 - 1}e_{\lambda'_1 - 1}^{(m-1)} & (-1)^{\lambda'_1}e_{\lambda'_1}^{(m-1)}  & \cdots & (-1)^{\lambda'_1 + G - 1}e_{\lambda'_1 + G - 1}^{(m-1)}  \\
    \vdots & \ddots & & \vdots \\
    (-1)^{\lambda'_G - G}e_{\lambda'_G - G}^{(m-1)}  & \cdots & & (-1)^{\lambda'_G}e_{\lambda'_G}^{(m-1)} \\
    1 & X & \cdots & X^G
\end{vmatrix}\\
\nonumber =& A_G X^G + A_{G-1} X^{G-1} + \cdots + A_0,
\end{align}
where
\begin{align}\label{def:AM}
A_G =&\begin{vmatrix}
    (-1)^{\lambda'_1 - 1}e_{\lambda'_1 - 1}^{(m-1)} & (-1)^{\lambda'_1}e_{\lambda'_1}^{(m-1)} & \cdots & (-1)^{\lambda'_1 + G - 2}e_{\lambda'_1 + G - 2}^{(m-1)} \\
    \vdots & \ddots & & \vdots \\
    (-1)^{\lambda'_G - G}e_{\lambda'_G - G}^{(m-1)} & \cdots & & (-1)^{\lambda'_G - 1}e_{\lambda'_G - 1}^{(m-1)}
\end{vmatrix}\\
\nonumber =&(-1)^{\lambda'_1 - 1} \cdots (-1)^{\lambda'_G - G} \begin{vmatrix}
    e_{\lambda'_1 - 1}^{(m-1)} & -e_{\lambda'_1}^{(m-1)}& \cdots & (-1)^{ G - 1}e_{\lambda'_1 + G - 2}^{(m-1)} \\
    \vdots & \ddots & & \vdots \\
    e_{\lambda'_G - G}^{(m-1)} & \cdots & & (-1)^{G-1}e_{\lambda'_G - 1}^{(m-1)}
\end{vmatrix}.
\end{align}

 Note that if $G=0$, then $P_G=1$ and we have a contradiction. Then, we shall assume that $G>0$.  Let $\boldsymbol{z}=(z_1, \ldots, z_{m-1})$ be any permutation of a subsequence of length $m-1$ of $(x_{\alpha_{1}}, \ldots, x_{\alpha_{m+G}})$. Let $\{t_1, \ldots t_{G+1}\}:=\{x_{\alpha_{1}}, \ldots, x_{\alpha_{m+G}}\}\setminus \{z_1, \ldots, z_{m-1}\}$ and set $\boldsymbol{y}_s:=(\boldsymbol{z}, t_{s})$ for $1\le s\le G+1$. Then $\boldsymbol{y}_1, \ldots, \boldsymbol{y}_{G+1}$ are permutations of subsequences of length $m$ of $(x_{\alpha_{1}}, \ldots, x_{\alpha_{m+G}})$ and by the conclusion above we have
\begin{equation}\label{PM1}
P_G(\boldsymbol{y}_{s}) = A_G(\boldsymbol{z}) {t_{s}}^G + A_{G-1}(\boldsymbol{z}) {t_{s}}^{G-1} + \cdots + A_0(\boldsymbol{z}) = 0
\end{equation}
for every $1\le s \le G+1$.

 Thus the polynomial $P_G(\boldsymbol{z}, X)$ of $\cfq[X]$, of degree at most $G$, has the $G+1$ distinct roots $t_1, \ldots, t_{G+1}$. It follows that $P_G(\boldsymbol{z}, X)$ is the zero polynomial. In particular $A_G(\boldsymbol{z})=0$.
By \eqref{def:AM}  we then have that
\begin{equation}\label{eq:det_1}
\det \big( e_{\lambda'_i-i+j-1}^{(m-1)}(\boldsymbol{z})\big)_{1\le i, j\le G}=\begin{vmatrix}
    e_{\lambda'_1 - 1}^{(m-1)} & e_{\lambda'_1}^{(m-1)} & \cdots & e_{\lambda'_1 + G - 2}^{(m-1)} \\
    \vdots & \ddots & & \vdots \\
    e_{\lambda'_G - G}^{(m-1)} & \cdots & & e_{\lambda'_G - 1}^{(m-1)}
\end{vmatrix}(\boldsymbol{z})=\pm A_G(\boldsymbol{z})=0
\end{equation}
for any permutation $\boldsymbol{z}$ of a subsequence of length $m-1$ of $(x_{\alpha_{1}}, \ldots, x_{\alpha_{m+G}})$. We assert that this is not possible.
Indeed, note that the determinant in \eqref{eq:det_1} is of the form
\begin{equation}\label{eq:det_claim}
\det \big( e_{t_i-i+j}^{(n)}\big)_{1\le i, j\le N}=\begin{vmatrix}
    e_{t_1}^{(n)} & e_{t_1+1}^{(n)} & \cdots & e_{t_1 + N - 1}^{(n)} \\
    \vdots & \ddots & & \vdots \\
    e_{t_N - (N-1)}^{(n)} & \cdots & & e_{t_N }^{(n)}
\end{vmatrix},
\end{equation}
where $n, N$ are positive integers and $n\geq t_1\geq \cdots t_N\geq 0$. Then our assertion and the lemma follows by the following claim.

\begin{claim}
Let $(y_1,\ldots,y_{n+N+1})$ be a sequence of $n+N+1$ pairwise
distinct nonzero elements in a field $K$. Then there exists an
$n$-subsequence $z=(z_1,\ldots,z_n)$ of
$(y_1,\ldots,y_{n+N+1})$ such that the determinant in (3.7)
does not vanish at $z$.
\end{claim}

\emph{Proof of Claim.} By induction on $N$.

Let $N=1$. Then the determinant in \eqref{eq:det_claim} is simply $e_{t_1}^{(n)}$. Let $\boldsymbol{z}=(z_1 \klk z_{n})$ be an $n$--subsequence of $(y_1 \klk y_{n+2})$. For $t_1=0$ we have that $e_{0}^{(n)}(\boldsymbol{z})=1\neq 0$. Let $t_1\geq 1$. If $n=t_1$, we have that $e_{n}^{(n)}(\boldsymbol{z})=z_1\cdot \ldots \cdot z_n\neq 0$. Now, let $n > t_1$ and toward a contradiction assume  that $e_{t_1}^{(n)}(\boldsymbol{z})=0$ for all $n$--subsequences $\boldsymbol{z}$  of $(y_1 \klk y_{n+2})$. Let $\boldsymbol{z}'=(z'_1 \klk z'_{n-1})$  be any $(n-1)$--subsequence of $(y_1 \klk y_{n+2})$.  Let $s_1,s_2$ be any two distinct elements of the set $\{y_1 \klk y_{n+1}\}\setminus \{z'_1 \klk z'_{n-1}\}$. Then $(\boldsymbol{z}', s_i)$ is a permutation of an $n$--subsequence of $(y_1 \klk y_{n+2})$ and by our assumption we have $e_{t_1}^{(n)}(\boldsymbol{z}', s_i)=0$  for $i=1,2$. By Lemma \ref{lowering the degree} it follows that $e_{t_1-1}^{(n-1)}(\boldsymbol{z}')=0$. We see that repeated application of  Lemma \ref{lowering the degree} leads to $e_{0}^{(n-t_1)}(\boldsymbol{w})=0$ with $\boldsymbol{w}$  a $(n-t_1)$--subsequence  of $(y_1 \klk y_{n+2})$. Since $e_{0}^{(n-t_1)}=1$, this is a contradiction. This  shows that the Claim holds for $N=1$.

Now, let $N> 1$ and assume that the Claim is true for $N-1$.

First suppose that $t_N=0$.
Then the last row of the determinant in \eqref{eq:det_claim} has a $1$ in its $N\times N$ entry and zeros elsewhere. Thus, the determinant in \eqref{eq:det_claim} is equal to the $(N-1) \times (N-1)$ determinant
$$
\det \big( e_{t_i-i+j}^{(n)}\big)_{1\le i, j\le N-1}=\begin{vmatrix}
    e_{t_1}^{(n)} & e_{t_1+1}^{(n)} & \cdots & e_{t_1 + N - 2}^{(n)} \\
    \vdots & \ddots & & \vdots \\
    e_{t_{N-1} - (N-2)}^{(n)} & \cdots & & e_{t_{N-1} }^{(n)}
\end{vmatrix}.
$$
By the inductive hypothesis this determinant does not vanish at all $n$-subsequences of $(y_1 \klk y_{n+N})$ and \emph{a fortiori} at all $n$-subsequences of $(y_1 \klk y_{n+N+1})$.

Next, suppose that $t_N\geq 1$.
First consider the case  $n=t_N$. Then $n=t_1=\cdots =t_N$ and the determinant in \eqref{eq:det_claim} is of the form
$$
\det \big( e_{n-i+j}^{(n)}\big)_{1\le i, j\le N}=\begin{vmatrix}
    e_{n}^{(n)} & e_{n+1}^{(n)} & \cdots & e_{n + N - 1}^{(n)} \\
    \vdots & \ddots & & \vdots \\
    e_{n - (N-1)}^{(n)} & \cdots & & e_{n }^{(n)}
\end{vmatrix}.
$$
This  determinant is lower triangular. Then it is equal to the product $(e_{n }^{(n)})^N$ of the elements in the main diagonal.
We have
$$
(e_{n }^{(n)})^N(\boldsymbol{z})=(z_1 \cdot \ldots \cdot z_n)^N\neq 0
$$
for any $n$--subsequence $\boldsymbol{z}$ of  $(y_1 \klk y_{n+N+1})$.

Then consider the case $n> t_N$. Toward a contradiction assume that the determinant in \eqref{eq:det_claim} vanish at all  $n$--subsequences $\boldsymbol{z}$  of  $(y_1 \klk y_{n+N+1})$.
Let $\boldsymbol{z}'=(z'_1, \ldots, z'_{n-1})$ be any  $(n-1)$--subsequence  of
$(y_1 \klk y_{n+N+1})$.  Let $s_1 \klk s_{N+1}$ be any $N+1$ elements of the set $\{y_1 \klk y_{n+N+1}\}\setminus \{z'_1, \ldots, z'_{n-1}\}$ and set  $\boldsymbol{z}_i=(\boldsymbol{z}',s_i)$ for $1\le i \le N+1$.  The $\boldsymbol{z}_i$ are permutations of $n$--subsequences of $(y_1 \klk y_{n+N+1})$. Thus, by our assumtion the determinant in \eqref{eq:det_claim} vanish at $\boldsymbol{z}_1 \klk \boldsymbol{z}_{N+1}$. By Lemma \ref{lowering the degree} it follows that
$$
\det \big( e_{(t_i-1)-i+j}^{(n-1)}\big)_{1\le i, j\le N}(\boldsymbol{z}')=\begin{vmatrix}
    e_{t_1-1}^{(n-1)} & e_{(t_1-1)+1}^{(n-1)} & \cdots & e_{(t_1-1) + N - 1}^{(n-1)} \\
    \vdots & \ddots & & \vdots \\
    e_{(t_N-1) - (N-1)}^{(n-1)} & \cdots & & e_{t_N-1 }^{(n-1)}
\end{vmatrix}(\boldsymbol{z}')=0
$$
for any  $(n-1)$--subsequence $\boldsymbol{z}'$  of
$(y_1 \klk y_{n+N+1})$.

Continuing in this way, by repeated application of Lemma \ref{lowering the degree}, we finally obtain that
$$
\det \big( e_{t_i-t_N-i+j}^{(n-t_N)}\big)_{1\le i, j\le N}(\boldsymbol{w})=\begin{vmatrix}
    e_{t_1-t_N}^{(n-t_N)} & e_{t_1-t_N+1}^{(n-t_N)} & \cdots & e_{t_1-t_N + N - 1}^{(n-t_N)} \\
    \vdots & \ddots & & \vdots \\
    e_{- (N-1)}^{(n-t_N)} & \cdots & & e_{0}^{(n-t_N)}
\end{vmatrix} (\boldsymbol{w})=0
$$
for any $(n-t_N)$--subsequence $\boldsymbol{w}$  of $(y_1 \klk y_{n+N+1})$. The last row of the previous determinant has a $1$ in its $N\times N$ entry and zeros elsewhere.  It follows  that
$$
\det \big( e_{t_i-t_N-i+j}^{(n-t_N)}\big)_{1\le i, j\le N-1}(\boldsymbol{w})=\begin{vmatrix}
    e_{t_1-t_N}^{(n-t_N)} & e_{t_1-t_N+1}^{(n-t_N)} & \cdots & e_{t_1-t_N + N - 1}^{(n-t_N)} \\
    \vdots & \ddots & & \vdots \\
    e_{t_{N-1}-t_{N}-N+2}^{(n-t_N)} & \cdots & & e_{t_{N-1}-t_N}^{(n-t_N)}
\end{vmatrix} (\boldsymbol{w})=0
$$
for any $(n-t_N)$--subsequence $\boldsymbol{w}$ of  $(y_1 \klk y_{n+N+1})$. Since $n-t_N$ is a positive integer and $n-t_N\geq t_1-t_N\geq \ldots \geq t_{N-1}-t_N\geq 0$, we have a contradiction with our inductive hypothesis and the Claim follows.
\end{proof}

We start by obtaining an upper bound for the dimension of set $C$ of \eqref{Conjunto C}.

\begin{lemma} \label{M=0}  Let the hypotheses as in Lemma \ref{cada m+M}.
 Then the set $C$ of \eqref{Conjunto C}  has dimension at most $m+G-1$. In particular, the singular locus of $\mathcal{V}$ has dimension at most $m+G-1$.
\end{lemma}

\begin{proof}
We consider the partition
$
C= \bigcup_{i=0}^n C_i,
$
where  $C_i$ is the set of points $\bfs{x}\in C $  having  exactly $i$ zero coordinates.
For $0\le i\le n$ we write
\[
C_i=\bigcup_{\substack{\bfs{\beta}=(\beta_1,\ldots,\beta_i)\\
1\le \beta_\ell\le n\\
\beta_\ell\neq \beta_h\ (\ell\neq h)}}
\bigl(C_i\cap \mathcal{L}_{\bfs{\beta}}\bigr),
\qquad
\mathcal{L}_{\bfs{\beta}}:=\{X_{\beta_1}=\cdots=X_{\beta_i}=0\}.
\]

We shall prove that $\dim C_i\cap \mathcal{L}_{\bfs{\beta}}\leq m+G-1$ for all $i$ and $\bfs{\beta}$, from which the lemma follows.

We consider the following two cases:

\medskip

\begin{enumerate}[
    label=\textbf{Case \arabic*:},
    leftmargin=0pt,
    labelsep=0pt,
    itemsep=1em,
    align=left,
    wide
]
\item \textbf{ $n-m-G+1 \leq i$.}
Then
\[
\dim C_i\cap \mathcal{L}_{\bfs{\beta}} \le \dim \mathcal{L}_{\bfs{\beta}} = n-i \le m+G-1.
\]

\medskip
\item \textbf{ $0\le i \le n-m-G$.}

Without loss of generality, assume that $\bfs{\beta}=(n-i+1,\ldots,n)$. Then the first $n-i$ coordinates $x_1, \ldots, x_{n-i}$ of any point $\bfs{x}\in C_i \cap \mathcal{L}_{\bfs{\beta}}$ are nonzero. Since $n-i\geq m+G$, the coordinates $x_1,\dots,x_{m+G}$ of $\bfs{x}$ are nonzero.
By Lemma \ref{cada m+M}, there exist distinct indexes $\alpha_1,\alpha_2\in \{1,\ldots, m+G\}$  such that $x_{\alpha_1}=x_{\alpha_2}$.
Next, if $n-i\geq m+G+1$, we exclude $x_{\alpha_1}$ and consider the $m+G$ nonzero coordinates $x_1,\dots, \widehat{x_{\alpha_1}},\ldots,x_{m+G+1}$ of $\bfs{x}$. By Lemma \ref{cada m+M}  there exist distinct   indexes $\alpha_3, \alpha_4 \in \{1,\ldots,{m+G+1}\}\setminus \{\alpha_1\}$ such that
$x_{\alpha_3} = x_{\alpha_4}$. After $n-i-m-G+1$ steps we obtain distinct indexes $\alpha_{2(n-i-m-G+1)-1}, \alpha_{2(n-i-m-G+1)}\in \{1,\ldots,n-i\}\setminus \{\alpha_1,\alpha_3,\ldots,\alpha_{2(n-i-m-G)-1}\}$ such that $$x_{\alpha_{2(n-i-m-G+1)-1}}=x_{\alpha_{2(n-i-m-G+1)}}.$$
This shows that $C_i\cap\mathcal{L}_{\bfs{\beta}}$ is contained in the (finite) union of linear varieties, each one of which is defined by a system of equations of the form
\begin{equation}\begin{cases}\label{ecuaciones iguales dos a dos}
    X_{\alpha_{2j-1}} = X_{\alpha_{2j}}, \quad\quad 1\leq j \leq n-i-m-G+1\\
    \\
    X_{n-i+1}=\cdots=X_n=0
\end{cases}.
\end{equation}
Here $\alpha_{2j-1}, \alpha_{2j}$, $1\leq j \leq n-i-m-G+1$, run over all possible choices of indexes in  $\{1, \ldots, n-i\}$ satisfying the conditions $\alpha_{2j-1}\neq \alpha_{2j}$ and $\alpha_{2j-1},\alpha_{2j}\in \{1, \ldots, m+G+j-1\}\setminus \{\alpha_{1}, \ldots, \alpha_{2j-3}\}$ for $1\leq j \leq n-i-m-G+1$.
Then by construction system \eqref{ecuaciones iguales dos a dos} is given by  $n-m-G+1$ independent equations and so it defines a variety of dimension $m+G-1$ in $\A^n$.  Hence $\dim(C_i\cap \mathcal{L}_{\bfs{\beta}})\le m+G-1$ and the lemma.
\end{enumerate}
\end{proof}

\begin{theorem} \label{teo: V es interseccion completa} Let the hypotheses as in Lemma \ref{cada m+M}.  If $2m+G-1\le n$ then the polynomials $f_1, \ldots, f_m$ form a regular sequence of $\fq[X_1,\ldots,X_n]$.  Further, if $2m+G\le n$ then $f_1, \ldots, f_m$ is a reduced regular sequence of $\fq[X_1,\ldots,X_n]$.
\end{theorem}

\begin{proof} Let $h_1, \ldots, h_m$ be the homogeneous components of higher degree $h_j=a_1X_1^{d_j}+ \cdots +a_nX_n^{d_j}$, $1\le j \le m$, of $f_1, \ldots, f_m$. We claim that $h_1, \ldots, h_m$ is a regular sequence of $\fq[X_1,\ldots,X_n]$. By Lemma \ref {lemma: reg seq and complete int} it is enough to show that the affine variety $V(h_1, \ldots, h_m)$ is a set--theoretic complete intersection of $\A^n$. To this aim consider any irreducible component  $V'$ of $V(h_1, \ldots, h_m)$. We have $\dim V'\geq n-m$ (see, e.g, \cite[Chapter 5, Theorem 3.4]{Kunz85}). To prove that $\dim V'\le n-m$ consider the set
$$
C':=\left\{\boldsymbol{x}\in V':\mathrm{rank}\, \mathrm{J}(h_1\klk h_m)(\boldsymbol{x})<m\right\}.
$$
Since the sequence $h_1 \klk h_m$ has the same degree pattern $(d_1 \klk d_m)$ as $f_1 \klk f_m$, by Lemma \ref{M=0} we deduce that $\dim C'\leq  m+G-1$.

Let $n\geq 2m+G-1$. If $C' \subsetneq V'$, then there exists some $\boldsymbol{p}\in V'$ with $\mathrm{rank}\, \mathrm{J}(h_1\klk h_m)(\boldsymbol{p})=m$.
As $T_{\boldsymbol{p}} V(h_1, \ldots, h_m)\subset \mathrm{ker}\left(\mathrm{J}(h_1\klk h_m)(\boldsymbol{p})\right)$
we deduce that $\dim T_{\boldsymbol{p}} V(h_1, \ldots, h_m)\leq n-m$.
Then $\dim V'\leq \dim_{\boldsymbol{p}}V(h_1, \ldots, h_m) \leq n-m$.  On the other hand, if $C' = V'$, then $\dim V'=\dim C'\leq m+G-1\leq n-m$. Thus we have  $\dim V'\le n-m$ in any case.  We conclude that, for $2m+G-1\leq n$, the affine variety $V(h_1, \ldots, h_m)$ is of pure dimension $n-m$, which proves our claim. By  Lemma \ref{lemma: comp hom de mayor grado suc reg} this implies  that $f_1,\ldots,f_m$ also is a regular sequence. This proves the first assertion of the theorem.

 Finally, let $2m+G\leq n$. It remains to prove that $(f_1\klk f_s)$ is a radical ideal for $1\le s \le m$. First, $(f_1)$ is a radical ideal because $f_1$ is square-free. Indeed if $f_1=g^2h$ with $g$ and $h$ in $\fq[X_1 \klk X_n]$ and $g$ irreducible, then $g$ divides both $\partial f_1/\partial X_1=d_1X_1^{d_1-1}$ and $\partial f_1/\partial X_2=d_1X_2^{d_1-1}$. This is impossible, since $d_1 \neq 0$ in $\fq$ by the hypothesis $p \nmid d_1$. Now, fix $s$ with $2\le s \le m$. In order to prove that $(f_1,\ldots,f_s)$ is a radical ideal, consider the subset  $C_s$ of the affine variety $V(f_1,\ldots,f_s)\subset \A^n$ given by
 $$
 C_s:=\left\{\boldsymbol{x}\in V(f_1,\ldots,f_s):\mathrm{rank}\, \mathrm{J}(f_1\klk f_s)(\boldsymbol{x})<s\right\}.
 $$
 Recall that $G=\sum_{j=2}^{m}g_j$, where $g_j=d_j-d_{j-1}-1\geq 0$ for $2\le j \le m$. Set $G_s:=\sum_{j=2}^{s}g_j$. As $G_s\le G$  the hypothesis $2m+G\le n$ implies that $2s+G_s\le n$. Thus, by Lemma \ref{M=0} we deduce that $\dim C_s\leq  s+G_s-1$.
 Therefore, for any irreducible component $W$ of the affine variety $V(f_1,\ldots,f_s)\subset \A^n$, we have that
$$
    \dim W-\dim C_s\cap W\geq \dim W-\dim C_s\geq n-s-(s+G_s-1)\geq 1.
$$
 In other words, $C_s$ has positive codimension in each irreducible component of $V(f_1,\ldots,f_s)$.
 By Theorem \ref{theorem: eisenbud 18.15} this implies  that $(f_1,\ldots,f_s)$  is a radical ideal, which concludes the proof.
\end{proof}

\begin{theorem} \label{dimension singular V M=0} Let the hypotheses as in Lemma \ref{cada m+M}. Assume further that in the polynomials of \eqref{eq:intro-system} we have  $a_i=1$ for $1\le i \le n$ and let $p> n-m-G+2$. Then the dimension of $C$ is at most $\frac{m-1}{2} +G$. In particular, the dimension of $\mathrm{Sing}(\mathcal{V})$, the singular locus of $\mathcal{V}$,
is at most $\frac{m-1}{2} +G$.
\end{theorem}
\begin{proof}  As in the proof of Lemma \ref{M=0} we consider the decomposition
\[
C = \bigcup_{i=0}^n \bigcup_{\substack{{\bfs{\beta}}=(\beta_1,\ldots,\beta_i)\\
1\le \beta_\ell\le n\\
\beta_\ell\neq \beta_h\ (\ell\neq h)}}
\bigl(C_i\cap \mathcal{L}_{\bfs{\beta}}\bigr),
\]
and prove that $\dim C_i\cap \mathcal{L}_{\beta}\leq \frac{m-1}{2} +G$ for all $i$ and ${\bfs{\beta}}$.

Fix $0\le i \le n$ and, without loss of generality, assume that ${\bfs{\beta}}=(n-i+1,\ldots, n)$.

Note that, as ${\mathcal{L}}_{{\bfs{\beta}}}=\{X_{n-i+1}=0, \ldots, X_{n}=0\}$, $C_i\cap {\mathcal{L}}_{{\bfs{\beta}}}$ is embedded  in the variety  $W \subset \A^{n-i}$ defined by the system:
\begin{equation} \label{sistema M=0}
\left\{
        \begin{array}{l}
          X_{1}^{d_1} + X_{2}^{d_1} + \cdots + X_{{n-i}}^{d_1} = b_1 \\
          \hspace{40pt} \vdots \\
              X_{1}^{d_m} + X_{2}^{d_m} + \cdots + X_{n-i}^{d_m}  = b_m.
        \end{array} \right.
\end{equation}

Observe that when $i=n$ or $i=n-1$, $\mathcal{V}\cap \mathcal{L}_\beta$ and a fortiori $C_i \cap \mathcal{L}_\beta$ is empty or zero-dimensional.

We consider the following three cases:
\begin{enumerate}[
    label=\textbf{Case \arabic*:},
    leftmargin=0pt,
    labelsep=0pt,
    itemsep=1em,
    align=left,
    wide
]

\item \leavevmode \,    $n-G\le i \le n-2$.

Since $W$ is a proper subvariety of $\A^{n-i}$, we have $$\dim C_i \cap \mathcal{L}_\beta \leq \dim W \leq n-i-1 \le G-1.$$

\item \leavevmode \,    $n-m-G+1\le i \le n-G-1$.

 Let $W'\subset \A^{n-i}$ be the variety defined by the first $\ell:=\lfloor \frac{n-i-G+1}{2} \rfloor$  equations of system \eqref{sistema M=0}.
 If $\ell=1$ then $W'$ is an hypersurface of $\A^{n-i}$ so, its dimension is $n-i-1=n-i-\ell$. Assume that $\ell \geq 2$.
  Set $G':=d_{\ell}-d_1-\ell+1$. Since $G'\leq G$, we have
   $$
   2\ell+G'-1\leq 2\ell +G-1\le n-i.
   $$
   By Theorem \ref{teo: V es interseccion completa}, it follows that $\dim W'=n-i-\ell$. Hence,
  \begin{alignat*}{3}
    \dim W'& = n-i- \left\lfloor\frac{n-i-G+1}{2}\right \rfloor \\
     & =\left \lceil\frac{n-i-G+1}{2}\right \rceil +G-1 \\
      & \leq \left \lceil\frac{m}{2}\right \rceil +G-1 \\
       & \leq \frac{m-1}{2} +G
  \end{alignat*}
Thus $$\dim C_i\cap {L} _{\beta}\le \dim W \le \dim W' \le \frac{m-1}{2} +G.$$

\item \leavevmode \,  $0\leq i \leq n-m-G$.

As in the previous case, we identify $C_i \cap {\mathcal{L}}_\beta$  with a subset of the variety $W \subset \A^{n-i}$ defined by \eqref{sistema M=0}. Thus, by \textbf{Case 2} of the proof of Lemma $\ref{M=0}$, $C _i \cap {\mathcal{L}}_\beta$ is included in the union of the  linear varieties of $\A^{n-i}$ of the form  $\mathcal{T}_{\alpha}=\{X_{\alpha_{2j-1}} -X_{\alpha_{2j}}=0,  1 \leq j \leq n - i - m - G + 1\}$,  with $\alpha_{2j-1}, \alpha_{2j}, 1 \leq j \leq n - i - m-G + 1$ any choice of indexes as in \eqref{ecuaciones iguales dos a dos}. Thus, it is enough to prove that $\dim C_i \cap {\mathcal{L}}_\beta \cap \mathcal{T}_{\alpha}\le \frac{m-1}{2}+G$ for any such variety $\mathcal{T}_{\alpha}$. Since $C_i \cap {\mathcal{L}}_\beta \cap \mathcal{T}_{\alpha}\subseteq W \cap \mathcal{T}_{\alpha}$ we shall show that $\dim W \cap \mathcal{T}_{\alpha} \le \frac{m-1}{2}+G$. For this purpose, we  substitute $X_{\alpha_{2j-1}}$ for $X_{\alpha_{2j}}$ in the equations of \eqref{sistema M=0} successively for all $1 \leq j \leq n - i - m-G + 1$.
Note that performing a substitution in an intermediate system has the effect of eliminating one variable and of adding $k\cdot 1$, for some $1\leq k\leq n-i-m-G+2$, to a coefficient of the system.

Since all coefficients in  system \eqref{sistema M=0} are equal to $1$ after performing these $n-i-m-G+1$ substitutions we obtain a system  of the form
\begin{equation} \label{sistema M=0 bis}
\left\{
        \begin{array}{l}
             c_1X_{\beta_1}^{d_1} + \cdots + c_{m+G-1}X_{\beta_{m+G-1}}^{d_1} = b_1 \\
             \hspace{40pt} \vdots \\
             c_1X_{\beta_1}^{d_m}+ \cdots + c_{m+G-1}X_{\beta_{m+G-1}}^{d_m}  = b_m,  \end{array} \right.
    \end{equation}
where $\beta_1, \ldots, \beta_{m+G-1}$ are distinct indexes of $\{1, \ldots, n - i\}$ and all coefficients $c_1, \ldots, c_{m+G-1}$ are integral multiples of $1$  of the form  $c_j=k_j \cdot 1$ with  $1\le k_j \le n-i-m-G+2$ for all $j$.  Since by hypothesis the characteristic $p$ of $\fq$ satisfies $p>n-m-G+2\geq n-i-m-G+2$, we deduce $c_j\neq 0$ for all $1\le j \le m+G-1$. Further note that by construction for each variable $X_j$, $1\le j \le n-i$, there exists a variable $X_{\beta_l}$ appearing in system \eqref{sistema M=0 bis} such that $X_j=X_{\beta_l}$ on $W \cap \mathcal{T}_{\alpha}$. This implies that $W \cap \mathcal{T}_{\alpha}$ is isomorphic to the variety $W'\subset \A^{m+G-1}$ defined by system \eqref{sistema M=0 bis}.
 We consider the variety $W''\subset \A^{m+G-1}$ defined by the first $\ell:=\left\lfloor \frac{m}{2} \right\rfloor$ equations of system \eqref{sistema M=0 bis}. If $\ell=1$, then $W''$ is a hypersurface of $\mathbb A^{m+G-1}$, and the desired bound follows. Hence, we may assume that $\ell\geq2$.
  Set $G'':=d_{\ell}-d_1-\ell +1$. Since $G''\leq G$, we have
 $$
 2\ell+G''-1\leq m+G-1.
 $$
  By Theorem \ref{teo: V es interseccion completa}, we deduce that
  $$\dim W''= (m+G-1)- \left\lfloor \frac{m}{2}\right\rfloor= \left\lceil \frac{m}{2}\right\rceil + G-1 \le  \frac{m-1}{2}+G.$$
  It follows that  $$\dim  W \cap \mathcal{T}_{\alpha} \leq  \dim W' \le \dim W'' \leq \frac{m-1}{2}+G,$$ which concludes the proof.

\end{enumerate}

\end{proof}

\begin{remark} \label{obeservacion 3. 6}
Under the hypotheses of Theorem \ref{dimension singular V M=0}, if moreover
$m \ge 2G+1,$
then
$
\dim C \le m-1.
$
Indeed, by Theorem \ref{dimension singular V M=0},
$
\dim C \le \frac{m-1}{2}+G \le m-1.$
In particular,
$
\dim \mathrm{Sing}(V) \le m-1.$
\end{remark}
\subsection{The geometry of the projective closure}\label{geo proyectiva}

In the following we analyze the geometric properties of the locus of $\mathrm{pcl}(\mathcal{V})$ at infinity, that is, of the  projective variety
\[
\mathcal{V}^{\infty} = \mathrm{pcl}(\mathcal{V}) \cap \{X_0 = 0\} \subset \mathbb{P}^{n-1}.
\]

Throughout this section  we assume that $a_i=1$ for $1\le i \le n$ in the polynomials $f_1 \klk f_m$ of \eqref{eq:intro-system} which define $\mathcal{V}$.
Thus the homogenizations  $f^{h}_1 \klk f^{h}_m$ of  $f_1 \klk f_m$ with respect to to a new variable $X_0$  are given by
\[
f_j^h := p^{(n)}_{d_j} - b_j X_0^{d_j}, \quad 1 \leq j \leq m,
\]
where
$
p^{(n)}_{d_j} = X_1^{d_j} + \cdots + X_n^{d_j}
$
is the symmetric power sum polynomial of order $d_j$.  We denote by $\mathrm{Sing} (\mathcal{V}^{\infty})$ the singular locus of $\mathcal{V}^{\infty}$.

  \begin{theorem}\label{teo:props_clausura_infinito}
Assume that $a_i=1$ for $1\le i \le n$ in the polynomials of \eqref{eq:intro-system}.  Further, suppose that the following hypotheses hold: the characteristic $p$ of $\fq$ does not divide any $d_i$ for $1\leq i \leq m$,  $p>n-m-G+2$ and $2m+G\leq n  $ if $m\geq 3$ and $n\geq 5+G$ if $m=2$. Then \( \mathcal{V}^{\infty}= \mathcal{V}_{\mathbb{P}^{n-1}}(p^{(n)}_{d_1},\ldots, p^{(n)}_{d_m}) \). Moreover $\mathcal{V}^{\infty}$ is an absolutely irreducible normal complete intersection of dimension $n - m - 1$ and degree $d_1 \cdots d_m$. Further, $\dim \mathrm{Sing} (\mathcal{V}^{\infty})\leq \frac{m-1}{2}+G-1.$
\end{theorem}

\begin{proof}
First note that $\mathcal{V}^{\infty}$ is embedded into $V_{\mathbb{P}^{n-1}}(p^{(n)}_{d_1},\ldots, p^{(n)}_{d_m})$ via the obvious isomorphism. We write $\mathcal{V}^{\infty}\subset V^*:=V_{\mathbb{P}^{n-1}}(p^{(n)}_{d_1},\ldots, p^{(n)}_{d_m})$. Let $\mathrm{J}(p^{(n)}_{d_1},\ldots, p^{(n)}_{d_m})$ denote the Jacobian matrix of $p^{(n)}_{d_1},\ldots, p^{(n)}_{d_m}$ with respect to the variables $X_1, \ldots, X_n$. Then, the singular locus $\mathrm{Sing}(V^*)$ of $V^*$ is contained in the subset $C^*$ of $V^*$ defined by
$$
C^*:=\{\boldsymbol{x}\in V^*: \mathrm{rank}\,  \mathrm{J}(p^{(n)}_{d_1},\ldots, p^{(n)}_{d_m})(\boldsymbol{x})< m\}
$$
and whose affine cone is the subset $C'$ of $V_{\A^{n}}(p^{(n)}_{d_1},\ldots, p^{(n)}_{d_m})$ defined by
$$
C'=\{\boldsymbol{x}\in V_{\A^{n}}(p^{(n)}_{d_1},\ldots, p^{(n)}_{d_m}): \mathrm{rank}\,   \mathrm{J}(p^{(n)}_{d_1},\ldots, p^{(n)}_{d_m})(\boldsymbol{x})< m\}.
$$
 By Theorem \ref{teo: V es interseccion completa} , we have that $V_{\A^{n}}(p^{(n)}_{d_1},\ldots, p^{(n)}_{d_m})$ is a complete intersection of dimension $n-m$. Thus $V^*$ is a complete intersection of dimension $n-m-1$. Then we have $\dim C'\leq   \frac{m-1}{2}+G$  by Theorem \ref{dimension singular V M=0}. Thus, $\dim C^*\le \frac{m-1}{2}+G-1$. It follows that $\mathrm{Sing}(V^*)$ has codimension at least $2$ in $V^*$. Thus $V^*$ is a normal variety. Then by Theorem  \ref{theorem: normal complete int implies irred}, $V^*$ is absolutely irreducible. On the other hand, by Theorem \ref{teo: V es interseccion completa}, $\mathcal{V}=V_{\A^{n}}(f_1, \ldots, f_m)$ has pure dimension $n-m$. Thus $\mathrm{pcl} (\mathcal{V})$ has positive dimension  $n-m$. It follows that $\mathcal{V}^{\infty}=\mathrm{pcl} (\mathcal{V}) \cap \{X_0=0\}$ has dimension at least $n-m-1$ (\cite[Chapter 9, Proposition 7]{CLO92}). Since $\mathcal{V}^{\infty}\subset V^*$ we deduce that $\dim \mathcal{V}^{\infty}=\dim V^*$. As $V^*$ is irreducible by Theorem \ref{theorem: normal complete int implies irred} we conclude that $\mathcal{V}^{\infty}=V^*$. Finally, by \cite[Section 2.5.6a]{Man2018} we have $\deg(\mathcal{V}^{\infty})=d_1\cdots d_m$.
\end{proof}

We call the set  $\Sigma^{\infty}:=\mathrm{Sing}(\mathrm{pcl}(\mathcal{V}))\cap \{X_0=0\}$ the singular locus of $\mathcal{V}$ at infinity.

\begin{lemma}\label{dim sing(pcl(V)) leq m+M divido 2}
    With the same hypotheses as in Theorem \ref{teo:props_clausura_infinito}, the singular locus $\Sigma^{\infty}$ of $\mathcal{V}$ at infinity has dimension at most $\frac{m-1}{2}+G-1.$
\end{lemma}

\begin{proof}
     By \cite[Lemma 1.1]{GhLa02a}, we have  $\Sigma^{\infty}\subset \mathrm{Sing} (\mathcal{V}^{\infty})$. Then the lemma follows by Theorem \ref{teo:props_clausura_infinito}.
\end{proof}

\begin{theorem} \label{teo: pcl V interseccion completa}
Let hypotheses as in Theorem \ref{teo:props_clausura_infinito}. Then
$\mathrm{pcl}(\mathcal{V})=V_{\Pp^n}(f_1^h,\ldots,f_m^h)$. Moreover,  $\mathrm{pcl}(\mathcal{V})$ is an absolutely irreducible, normal complete intersection  of dimension $n-m$ and degree $d_1\cdots d_m$.  Further,  $\dim \mathrm{Sing} (\mathrm{pcl}(\mathcal{V}))\leq\frac{m-1}{2}+G$.
\end{theorem}

\begin{proof}
Set $W:=V_{\Pp^n}(f_1^h,\ldots,f_m^h)$. We then have an isomorphism of projective varieties  $W\cap \{X_0=0\}\cong V_{\Pp^{n-1}}(p^{(n)}_{d_1},\ldots, p^{(n)}_{d_m})$.  By Theorem \ref{teo:props_clausura_infinito}, $V^{\infty}=V_{\Pp^{n-1}}(p^{(n)}_{d_1},\ldots, p^{(n)}_{d_m})$, and therefore $W\cap \{X_0=0\}$, has pure dimension $n-m-1$. Since $W\cap \{X_0=0\}$ is defined by the $m+1$ polynomials $f_1^h,\ldots,f_m^h, X_0$, it follows that $W\cap \{X_0=0\}$ is a set-theoretic complete intersection of dimension $n-m-1$. By Lemma \ref{lemma: reg seq and complete int} we conclude that $f_1^h,\ldots,f_m^h, X_0$ is a regular sequence in $\fq[X_0, \ldots, X_n]$. Thus $f_1^h,\ldots,f_m^h$ is a regular sequence in $\fq[X_0, \ldots, X_n]$ and by Lemma \ref{lemma: reg seq and complete int} we conclude that $W$ is a set-theoretic complete intersection of dimension $n-m$.

 Next, let $\mathrm{J}(f_1^h,\ldots,f_m^h)$ be the Jacobian matrix of $f_1^h,\ldots,f_m^h$ with respect to the variables $X_0, \ldots, X_n$.
 Then $\mathrm{Sing} (W)$ is contained in the subset $\widetilde{C}$  of $W$ defined by
$$
\widetilde{C}:=\{\boldsymbol{x}\in W: \mathrm{rank}\,   \mathrm{J}(f_1^h,\ldots,f_m^h)(\boldsymbol{x})< m\}.
$$

We claim that $\dim \widetilde{C} \leq \frac{m-1}{2}+G$. We will show that $\dim \widetilde{C}\cap \{X_0\neq 0\}\leq \frac{m-1}{2}+G$ and $\dim \widetilde{C}\cap \{X_0= 0\}\leq \frac{m-1}{2}+G-1$, from which our claim follows. For this purpose, consider the Jacobian matrix $\mathrm{J}(f_1,\ldots,f_m)$. By using the Euler's Homogeneous Function Theorem it is easy to see that for any $\boldsymbol{x}=(1:x_1:\,\dots\, :x_n)\in W\cap \{X_0\neq 0\}$ it holds $\mathrm{rank}\,   \mathrm{J}(f_1^h,\ldots,f_m^h)(\boldsymbol{x})=\mathrm{rank}\,   \mathrm{J}(f_1,\ldots,f_m)(x_1, \ldots, x_n)$. If $C:=\{(x_1, \ldots, x_n)\in V_{\A^n}(f_1,\ldots,f_m): \mathrm{rank}\,   \mathrm{J}(f_1,\ldots,f_m)(\boldsymbol{x})< m\}$, then we have an isomorphism of quasi projective varieties $\widetilde{C}\cap  \{X_0\neq 0\} \cong C$. Thus $\dim \widetilde{C}\cap \{X_0\neq 0\}\leq \frac{m-1}{2}+G$ by Theorem \ref{dimension singular V M=0}. Then, consider the subset $C'$ of $V_{\A^{n}}(p^{(n)}_{d_1},\ldots, p^{(n)}_{d_m})$ defined by
$$
C'=\{(x_1, \ldots, x_n)\in V_{\A^{n}}(p^{(n)}_{d_1},\ldots, p^{(n)}_{d_m}): \mathrm{rank}\,   \mathrm{J}(p^{(n)}_{d_1},\ldots, p^{(n)}_{d_m})(x_1, \ldots, x_n)< m\},
$$
where $\mathrm{J}(p^{(n)}_{d_1},\ldots, p^{(n)}_{d_m})$ denote the Jacobian matrix of $p^{(n)}_{d_1},\ldots, p^{(n)}_{d_m}$ with respect to to the variables $X_1, \ldots, X_n$.
Then for any point $\boldsymbol{x}=(x_0:x_1:\, \dots\, :x_n)\in W$ it holds $\mathrm{rank}\,   \mathrm{J}(p^{(n)}_{d_1},\ldots, p^{(n)}_{d_m})(x_1, \ldots, x_n)\leq \mathrm{rank}\,   \mathrm{J}(f_1^h,\ldots,f_m^h)(\boldsymbol{x})$. Thus, the affine cone $\mathcal{C}\subset \A^{n+1}$ of the projective variety $\widetilde{C}\cap \{X_0= 0\}$ satisfies $\mathcal{C}\subset \{0\}\times C'$. Since $\dim C' \leq \frac{m-1}{2}+G$ by Theorem \ref{dimension singular V M=0}, we conclude that $\dim \widetilde{C}\cap \{X_0= 0\}=\dim \mathcal{C}-1\leq \frac{m-1}{2}+G-1$. This proves our claim.

The previous claim implies that $\widetilde{C}$ has codimension $\geq 2$ in $W$.
Now,  consider the subset $C^a$ of the affine variety $V_{\A^{n+1}}(f_1^h,\ldots,f_m^h)$ defined by
$$
C^a:=\{\boldsymbol{x}\in V_{\A^{n+1}}(f_1^h,\ldots,f_m^h): \mathrm{rank}\,   \mathrm{J}(f_1^h,\ldots,f_m^h)(\boldsymbol{x})< m\}.
$$
Since  $V_{\A^{n+1}}(f_1^h,\ldots,f_m^h)$ and $C^a$ are the affine cones of $V_{\Pp^n}(f_1^h,\ldots,f_m^h)$ and $\widetilde{C}$ respectively, we deduce that $C^a$ has codimension $\geq 2$ in $V_{\A^{n+1}}(f_1^h,\ldots,f_m^h)$. By Theorem \ref{theorem: eisenbud 18.15} it follows that $(f_1^h,\ldots,f_m^h)$ is a radical ideal. This completes the proof that  the sequence  $f_1^h,\ldots,f_m^h$ defines a complete intersection in $\Pp^n$. Further since $\mathrm{Sing} (W)$ has codimension $\geq 2$ in $W$ it follows that  $W$ is a normal variety and, by Theorem \ref{theorem: normal complete int implies irred}, that $W$ is an absolutely irreducible variety of dimension $n-m$. Since $\mathrm{pcl}(\mathcal{V})\subseteq W= V_{\Pp^n}(f_1^h,\ldots,f_m^h)$ and $\dim \mathrm{pcl}(\mathcal{V})=\dim \mathcal{V}=n-m$ we conclude that $\mathrm{pcl}(\mathcal{V})= V_{\Pp^n}(f_1^h,\ldots,f_m^h)$. Finally, by \cite[Theorem 18.3]{Harris92} we have $\deg(\mathrm{pcl}(\mathcal{V}))=d_1\cdots d_m$.
\end{proof}
\begin{remark} \label{obeservacion 3.10}
Under the hypotheses of Theorems \ref{teo:props_clausura_infinito}  and \ref{teo: pcl V interseccion completa}, if moreover $
m\ge 2G+1,$
then the corresponding bounds for the dimensions of the singular loci can be improved to
$
\dim \mathrm{Sing}(\mathcal{V}^\infty)\le m-1$ and
$\dim \mathrm{Sing}(\operatorname{pcl}(\mathcal{V}))\le m-1.$
Similarly, under the hypotheses of Lemma \ref{dim sing(pcl(V)) leq m+M divido 2},
$
\dim \Sigma^{\infty}\le m-1.$
Indeed, these inequalities follow immediately from the improved estimate
$
\dim C\le m-1$
obtained after Theorem \ref{dimension singular V M=0}.
\end{remark}

Notice that, setting $a_i=1$ for $1\leq i\leq n$
in the polynomials of \eqref{eq:intro-system}, we obtain systems
defined by the polynomials
$p_{d_1}^{(n)}-b_1,\ldots,p_{d_m}^{(n)}-b_m
\in\mathbb{F}_q[X_1,\ldots,X_n]$, where
\[
p_{d_j}^{(n)}=X_1^{d_j}+\cdots+X_n^{d_j},
\qquad 1\leq j\leq m.
\]
Therefore, Theorems \ref{teo: V es interseccion completa}, \ref{teo:props_clausura_infinito} and \ref{teo: pcl V interseccion completa} yield,
as a particular case, the regularity, radicality,
and normal complete intersection properties
stated in Theorem \ref{thm:intro-geom}.

Furthermore, assume that $m\geq 2G+1$ and
$p>n-m-G+2$. Then, by Remarks \ref{obeservacion 3. 6} and \ref{obeservacion 3.10}, the
corresponding bounds for the dimensions of the
singular loci can be improved to $m-1$.
By incorporating these improved bounds into the proofs
of Theorems \ref{teo: V es interseccion completa}, \ref{teo:props_clausura_infinito} and \ref{teo: pcl V interseccion completa}, the conditions on the
number of variables $n$ can be relaxed accordingly.
More precisely,
one obtains the same numerical conditions on $n$ as
in the case $G=0$. Namely, $n\geq 2m-1$,
$n\geq 2m$, and $n\geq 2m+1$ yield, respectively,
the regular sequence, radicality, and normal projective
complete intersection conclusions.

These results may be compared with those of Conca
et al.~\cite{CSS2024}, who studied ideals generated by
 polynomials $p_{d_j}^{(n)}$, $1\leq j \leq n$, over $\mathbb{C}$. They proved
that, for distinct exponents $d_1,\ldots,d_m$, the
corresponding power sums generate a complete intersection ideal of codimension \(m\)
when $n\geq 2m-1$, define a radical ideal when
$n\geq 2m$, and a normal domain when $n\geq 2m+1$.

Our results establish analogous properties over finite
fields. In particular, under our hypotheses on the
characteristic $p$ and on the exponents
$d_1,\ldots,d_m$, we obtain the same thresholds for
the corresponding regular sequence, radicality,
and normality properties. The main difference is that,
over finite fields, additional assumptions are required
on $p$ and on the parameter $G=d_m-d_1-m+1$.
Thus, our results give finite-field analogues of the
results of Conca et al. for the families of power sums
satisfying these explicit hypotheses.

\section{Estimates on the number of $\fq$--rational points  of power sums varieties}

The geometric properties established in the previous sections provide the structural framework needed
for the arithmetic analysis that follows. In particular, regularity, complete intersection
and normality allow us to apply point-counting estimates over finite fields to the varieties
defined by symmetric power sum systems. In this section, we exploit these properties to
derive explicit bounds for the number of $\mathbb{F}_q$--rational points and effective
existence results.

In what follows, we shall use an estimate on the number of
$\fq$--rational points of a projective complete intersection due to S. Ghorpade
and G. Lachaud (\cite{GhLa02a}; see also \cite{GhLa02}). Let $V\subset \Pp^N$ an $\fq$--irreducible complete intersection of dimension $r$. Let $I(V)=(f_1\klk f_{N-r})$ and $d_i=deg(f_i)$ for $1\le i\le n-r$. Note that the degrees $d_1,\ldots, d_{N-r}$ depend only on
$V$ and not on the system of generators. Arranging the $d_i$ in such
a way that $d_1\geq d_2 \geq \cdots \geq d_{N-r}$, we call $(d_1,\ldots,
d_{N-r})$ the {\em multidegree} of
$V$.

In \cite[Theorem 6.1]{GhLa02a}, the authors prove that,
for an irreducible $\fq$--complete intersection $V\subset \Pp^N$
of dimension $r$, multidegree ${\bfs{d}}=(d_1,\ldots,d_{N-r})$ and singular locus of dimension at most $s$ with $0\leq s\leq r-1$, the number $|V(\fq)|$ of $\fq$--rational points of $V$ satisfies the estimate.

\begin{equation}\label{eq: estimacion de Ghorpade Lachaud}
\big||V(\fq)|-p_r\big|\leq b_{r-s-1}'(N-s-1,{\bfs{d}})\, \, q^{\frac{r+s+1}{2}}+C_s(V)q^{\frac{r+s}{2}},
\end{equation}
where $p_r:=q^r+q^{r-1}+\cdots+1$,  $b_{r-s-1}'(N-s-1,{\bfs{d}})$ is the  $(r-s-1)$--th primitive Betti of a nonsingular complete intersection in $\Pp^N$
of dimension $r-s-1$ and  multidegree ${\bfs{d}}$,
and $C_s(V):=\sum_{i=r}^{r+s}b_{i,\ell}(V)+\varepsilon_i$, where
$b_{i,\ell}(V)$ denotes the $i$--th $\ell$--adic Betti number of $V$ for a prime $\ell$ different from $p:=\mathrm{char}(\fq)$ and
$\varepsilon_i:=1$ for even $i$ and $\varepsilon_i:=0$ for odd $i$.
From \cite[Proposition 4.2]{GhLa02a}
\begin{equation}\label{cota de Betti}
b_{r-s-1}'(N-s-1,{\bfs{d}})\leq \binom{N-s}{r-s-1}\cdot (d+1)^{N-s-1},
\end{equation}
where $d:=\max\{d_1,\ldots,d_{N-r}\}$.
On the other hand, from \cite[Theorem 6.1]{GhLa02a}, we have
\begin{equation}\label{constante C}
C_s(V)\leq 9\cdot 2^{N-r}\cdot ((N-r)d+3)^{N+1}.
\end{equation}

Now we consider $\mathcal{V}$, the variety defined by $f_j=X_1^{d_j}+\cdots+X_n^{d_j}-b_j$ for $1\le j\le m$ with $b_j\in \F_q$ and recall that $G=d_m-d_1-m+1$. Assume that $G>0$ when $m=2$. 

Denote by  $\mathrm{pcl}(\mathcal{V})(\fq)$ the set of $\fq$--rational points of $\mathrm{pcl}(\mathcal{V})$.
We also define:
\begin{equation}
s:=\left\lfloor\frac{m-1}{2}\right\rfloor+G.
\end{equation}

Let $s^\infty=s-1$. It's easy to check that with the hypothesis of Theorem \ref{teo:props_clausura_infinito} $s\le \dim \mathrm{pcl}(\mathcal{V})-1$ and $s^\infty\le \dim \mathcal{V}^\infty-1$ so that the hypothesis of S. Ghorpade and G. Lachaud are satisfied. From Theorem \ref{teo: pcl V interseccion completa}, the estimate \eqref{eq: estimacion de Ghorpade Lachaud}, the bounds \eqref{cota de Betti} and \eqref{constante C}, we deduce that
\begin{equation}\label{estimation pcl V}
    \big||\mathrm{pcl}(\mathcal{V})(\mathbb{F}_q)|-p_{n-m}\big|\leq{n-s\choose n-s-m-1}{(d_m+1)}^{n-s-1}q^{\frac{n-m+s+1}{2}}+9\cdot 2^{m}(md_m+3)^{n+1}q^{\frac{n-m+s}{2}}.
\end{equation}

Now we estimate the number of $\fq$--rational points of $\mathcal{V}^{\infty}\subset \Pp^{n-1}$. From Theorem \ref{teo:props_clausura_infinito}, the bounds \eqref{cota de Betti} and \eqref{constante C},  we obtain
\begin{equation} \label{estimacion pcl(V) infinito}
    \big||\mathrm{\mathcal{V} ^{\infty}}(\fq)|-p_{n-m-1}\big|\leq{n-s\choose n-s-m-1}{(d_m+1)}^{n-s-1}q^{\frac{n-m+s-1}{2}}+9\cdot 2^{m}(md_m+3)^{n}q^{\frac{n-m+s-2}{2}}.
\end{equation}
With these results proven, we can now state our main theorem.
\begin{theorem}\label{teo: estimacion ptos racionales V}.
With the same hypotheses as Theorem \ref{teo:props_clausura_infinito}. Let $|\mathcal{V}(\mathbb{F}_q)|$ be the number of $\mathbb{F}_q$--rational points of the variety $\mathcal{V}$. Then
\begin{equation*}
\left|\,|\mathcal{V}(\mathbb{F}_q)| - q^{n - m}\,\right| \leq 27 \cdot 2^{m - 1} (3 + d_m m)^{n + 1} \cdot q^{\frac{n-m+s+1}{2}}.
\end{equation*}
\end{theorem}
\begin{proof}
From \eqref{estimation pcl V} and \eqref{estimacion pcl(V) infinito} and the fact that $|\mathrm{pcl}(\mathcal{V})(\fq) \setminus \mathcal{V}^{\infty}(\fq)|=|\mathcal{V}(\fq)|$, we obtain that
\begin{align*}
    \big||\mathcal{V}(\fq)|-q^{n-m}\big|&\leq
    {n-s\choose n-s-m-1}{(d_m+1)}^{n-s-1}q^{\frac{n-m+s+1}{2}}+9\cdot 2^{m}(md_m+3)^{n+1}q^{\frac{n-m+s}{2}}\\
    &+ {n-s\choose n-s-m-1}{(d_m+1)}^{n-s-1}q^{\frac{n-m+s-1}{2}}+9\cdot 2^{m}(md_m+3)^{n}q^{\frac{n-m+s-2}{2}}\\\\
\end{align*}
Taking into account that $\binom{a}{b} \leq 2^a$, we obtain
\begin{align*}
    \big||\mathcal{V}(\fq)|-q^{n-m}\big|&\leq2^{n-s}{(d_m+1)}^{n-s-1}q^{\frac{n-m+s+1}{2}}+9\cdot 2^{m}(md_m+3)^{n+1}q^{\frac{n-m+s}{2}}\\
    &+ 2^{n-s}{(d_m+1)}^{n-s-1}q^{\frac{n-m+s-1}{2}}+9\cdot 2^{m}(md_m+3)^{n}q^{\frac{n-m+s-2}{2}}\\
    & \leq 2^{n-s}(d_m+1)^{n-s-1}q^{\frac{n-m+s-1}{2}}(1+q)\\
    &+9\cdot2^{m}(md_m+3)^{n+1}q^{\frac{n-m+s-2}{2}}(1+q)\\
    & = 2 (2d_m+2)^{n-s-1}q^{\frac{n-m+s-1}{2}}(1+q)\\
    &+9\cdot2^{m}(md_m+3)^{n+1}q^{\frac{n-m+s-2}{2}}(1+q)\\
    &\leq 9\cdot2^{m}(md_m+3)^{n+1}q^{\frac{n-m+s-2}{2}}\Big(1+\frac{q^{\frac{1}{2}}}{9 \cdot 2^{m-1}}\Big)(1+q)\\
    &\leq 9\cdot2^{m}(md_m+3)^{n+1}q^{\frac{n-m+s-2}{2}}\frac{3}{2}q^{\frac{3}{2}}\\
    &\leq 27\cdot2^{m-1}(md_m+3)^{n+1}q^{\frac{n-m+s+1}{2}},
\end{align*}

which concludes our main result.

\end{proof}

\begin{remark}
    In the particular case that $m=2$ and $G=0$ the variety $\mathcal{V}^{\infty}$ is nonsingular (see Theorem \ref{teo:props_clausura_infinito}) and we got the following estimation due to Deligne \cite{Del74},
  \begin{align*}
    \big||\mathcal{V}^\infty(\fq)|-p_{n-m-1}\big|&\le b'_{n-m-1}(n-1,\boldsymbol{d}) q^{\frac{n-m-1}{2}}\\
    &\le {n\choose n-m-1}(d_m+1)^{n-1} q^{\frac{n-m-1}{2}}.
    \end{align*}
\end{remark}

\begin{remark}
We shall compare the estimate of Theorem \ref{teo: estimacion ptos racionales V} with the one obtained in \cite{GottigPerezPrivitelli2024} for the same family of power-sum systems. Under the hypotheses considered there, the authors prove
$
\bigl||\mathcal{V}(\mathbb F_q)|-q^{n-m}\bigr|
\leq
27\cdot 2^{m-1}(3+md_m)^{n+1}q^{n/2}.
$
In the present paper, Theorem~4.1 gives
$
\bigl||\mathcal{V}(\mathbb F_q)|-q^{n-m}\bigr|
\leq
27\cdot 2^{m-1}(3+md_m)^{n+1}
q^{(n-m+s+1)/2},
$
where $s=(m-1)/2+G$ and $G=d_m-d_1-m+1$. Hence,
$
\frac{n-m+s+1}{2}
=
\frac{n}{2}-\frac{m-1-2G}{4}.
$
Therefore, whenever $2G<m-1$, the estimate obtained here improves the exponent of $q$ in \cite{GottigPerezPrivitelli2024} by the factor
$
q^{-(m-1-2G)/4}.
$
In particular, when $G=0$, that is, when the exponents are consecutive, the error term improves from order $q^{n/2}$ to order
$
q^{\,n/2-(m-1)/4}.
$
If $2G=m-1$, both estimates have the same exponent of $q$, whereas if $2G>m-1$, the estimate of \cite{GottigPerezPrivitelli2024} has the better exponent. The improvement obtained here is a consequence of the sharper bound
$
\dim\operatorname{Sing}(\operatorname{pcl}(\mathcal{V}))\leq (m-1)/2+G,
$
compared with the bound $m-1$ used in \cite{GottigPerezPrivitelli2024}. We emphasize, however, that the present estimate is obtained under additional restrictions on the characteristic and on the parameters. Thus, the two results are complementary: the earlier estimate applies under weaker hypotheses, while the present one yields a genuine power saving in $q$ when the gaps between the exponents are sufficiently small.
\end{remark}

\begin{remark} \label{rem:comparison-podesta-videla}
It is useful to compare our estimate with the recent Weil-sum bound obtained by
Podestá and Videla in \cite[Corollary 4.2]{PV2026}.
For the system
$$
X_1^{d_i}+\cdots+X_n^{d_i}=b_i,
\qquad 1\leq i\leq m,
$$
their result gives
$\Big ||\mathcal{V}(\fq)|-q^{n-m}\Big|
\leq
(q^m-1)q^{\frac n2-m}(d_m-1)^n,$
provided that, for every
$\boldsymbol{\alpha}=(\alpha_1,\ldots,\alpha_m)\neq 0$, the polynomial
$f_{\boldsymbol{\alpha}}(T)
=
\alpha_1T^{d_1}+\cdots+\alpha_mT^{d_m}$
is non-degenerate, in the sense that it cannot be written as
$g^p-g$ for some $g\in\mathbb F_q[T]$.

According to the hypotheses of Theorem \ref{teo:props_clausura_infinito}, our estimate has the form
$$
\big||\mathcal{V}(\fq)|-q^{n-m}\big|
\leq
27\cdot 2^{m-1}(3+md_m)^{n+1}
q^{\frac{n-m+s+1}{2}}
=
27\cdot 2^{m-1}(3+md_m)^{n+1}
q^{\frac{n}{2}-\frac{m-1-2G}{4}}.
$$
A practical limitation of the Weil-sum estimate is that its applicability
requires the non-degeneracy of every polynomial
$f_{\boldsymbol{\alpha}}$, $\boldsymbol{\alpha}\neq \boldsymbol{0}$,
a condition which may be nontrivial to verify in general. In our setting,
the explicit assumption $p\nmid d_i$ for all $i$ guarantees these
non-degeneracy conditions simultaneously. In the other hand this condition is necessary for the radicality of the ideals generated by our power sum polynomials. Thus, while their hypothesis is
formally more general with respect to the characteristic, ours provides a
simple and readily checkable structural criterion.

On the other hand, whenever $2G<m-1,$ our geometric approach yields an error term of order $q^{{\frac{n}{2}-\frac{m-1-2G}{4}}}$ for fixed $m,n$ and degrees.
The Weil-sum estimate provides an error term of  order  $q^{\frac{n}{2}}$  our estimate yields a genuine power
saving in $q$.

The Weil-sum bound has a substantially better dependence on the number of variables
and on the degrees, through the factor $(d_m-1)^n$, whereas our constant contains
$(3+md_m)^{n+1}$.
This comparison highlights the different nature of the two approaches.
The character-sum argument estimates separately the one-variable polynomials
$f_{\boldsymbol{\alpha}}$, whereas our method exploits the global geometry of
the complete intersection, and in particular the dimension of its singular
locus. It is precisely this additional geometric information that produces
the improvement in the exponent of $q$.
\end{remark}

Combining the geometric properties established in Section 4 with the point-counting estimate of Theorem \ref{teo: estimacion ptos racionales V}, we arrive at the following existence result.

\begin{theorem}\label{prop: existencia}
With the same hypotheses as Theorem \ref{teo:props_clausura_infinito}.
If
\(
q>\Bigl(27\cdot 2^{m-1}\cdot (3+ md_m)^{n+1}\Bigr)^{\frac{2}{\,n-m-s-1\,}},
\)
then $|\mathcal{V}(\F_q)|>0$. In particular, for fixed \(m\) and fixed degrees \(d_1,\ldots,d_m\), if \(n\) is sufficiently large, then the existence of an \(\mathbb F_q\)--rational point is guaranteed whenever
\(
q>(3+md_m)^2
\)

Indeed, the threshold
\(
\left(
27\cdot 2^{m-1}(3+md_m)^{n+1}
\right)^{\frac{2}{n-m-s-1}}
\)
converges to
\(
(3+md_m)^2
\)
as \(n\to\infty\).

\end{theorem}
\begin{remark}
   Proposition \cite[Proposition 3.9]{GottigPerezPrivitelli2024}
provides the sufficient condition
$
q>
\left(
27\cdot 2^{m-1}(3+md_m)^{n+1}
\right)^{\frac{2}{n-2m}}
$
for the existence of an $\mathbb F_q$--rational point. In the present paper, Theorem \ref{prop: existencia} gives the condition
$
q>
\left(
27\cdot 2^{m-1}(3+md_m)^{n+1}
\right)^{\frac{2}{n-m-s-1}},
$
where $s=(m-1)/2+G$ and $G=d_m-d_1-m+1$. Since
$
n-m-s-1
=
n-2m+\frac{m-1-2G}{2},
$
whenever $2G<m-1$ we have
$
n-m-s-1>n-2m,
$
and hence the existence condition obtained here is strictly weaker than the one in \cite{GottigPerezPrivitelli2024}. Thus, in this range, the present result guarantees the existence of $\mathbb F_q$-rational solutions for smaller values of $q$. In the case $2G=m-1$, both criteria have the same exponent, while for $2G>m-1$ the earlier result gives the better exponent. We emphasize that the present criterion requires the additional hypotheses of Theorem \ref{teo:props_clausura_infinito}, so the two results are complementary. Moreover, for fixed $m$ and fixed degrees, both criteria have the same asymptotic threshold as $n\to\infty$.
\end{remark}

\section{A generalized Waring problem over finite fields}

Waring's problem over finite fields has been extensively studied for a single diagonal equation, where the main question is to determine the minimum number of variables required to represent every element of the field as a sum of $d$--th powers. Generalizations to systems of diagonal equations have also been considered, although the available results mainly concern a small number of equations or particular families of exponents. For instance, Castro et al. \cite{CRGF08} study a two-equation version of the problem and determine, in several special cases, the minimum number of variables ensuring solvability for every pair of prescribed right-hand sides. In contrast, we consider systems of an arbitrary number $m$ of power-sum equations with distinct, not necessarily consecutive, exponents $d_1<\cdots<d_m$.   More
precisely, given positive integers which are not divisible by $p$,
$1\leq d_1<\cdots<d_m,$ $m\geq 3$,
we consider the system
\begin{equation} \label{eq:waring}
\begin{cases}
X_1^{d_1}+\cdots+X_n^{d_1}=b_1,\\
\hspace{1.4cm}\vdots\\
X_1^{d_m}+\cdots+X_n^{d_m}=b_m,
\end{cases}
\end{equation}
where $(b_1,\ldots,b_m)\in\mathbb{F}_q^m$. In this setting, we define $\gamma(d_1,\ldots,d_m,q)$ as the least integer $n$
such that system \eqref{eq:waring} admits an $\mathbb F_q$-rational solution for every
$(b_1,\ldots,b_m)\in\mathbb F_q^m$, whenever such an integer exists. We now consider values of $n$ satisfying the hypotheses of Theorem \ref{teo: estimacion ptos racionales V} in order
to obtain an explicit upper bound for $\gamma(d_1,\ldots,d_m,q)$.
Since the estimate of Theorem \ref{teo: estimacion ptos racionales V} is uniform in the choice of
$(b_1,\ldots,b_m)$, it yields an explicit criterion ensuring that system \eqref{eq:waring}
is solvable for every $(b_1,\ldots,b_m)\in\mathbb F_q^m$.

Let $N$  be the number of $\fq$--rational solutions of  \eqref{eq:waring}.
From Theorem \ref{teo: estimacion ptos racionales V} we have that
\ \[
    \left|N-q^{n-m}\right|
    \leq
    27\cdot 2^{m-1}(3+d_m m)^{n+1}
    q^{\frac{n-m+s+1}{2}},
    \]
    where $s=\frac{m-1}{2}+G.$

Hence, we have that
$$
N\geq q^{n-m}
-
27\cdot 2^{m-1}(3+md_m)^{n+1}
q^{\frac{n-m+s+1}{2}}.
$$
Therefore, $N>0$ provided that
\[
q^{\frac{n-m-s-1}{2}}
>
27\cdot 2^{m-1}(3+md_m)^{n+1}.
\]

Since
$27\cdot 2^{m-1}(3+md_m)^{n+1}
\leq
\left(\frac{7}{2}md_m\right)^{n+1},$
it is enough to require
$$q^{\frac{n-m-s-1}{2}}
>
\left(\frac{7}{2}md_m\right)^{n+1}.$$
Assume now that
$q>\left(\frac{7}{2}md_m\right)^2.$
Taking logarithms, the previous inequality is equivalent to
\[
\frac{n-m-s-1}{2}\log q
>
(n+1)\log\left(\frac{7}{2}md_m\right),
\]
and hence
\[
n
\left(
\frac{1}{2}\log q
-
\log\left(\frac{7}{2}md_m\right)
\right)
>
\log\left(\frac{7}{2}md_m\right)
+
\frac{m+s+1}{2}\log q.
\]
Thus,
\[
n>
\frac{
\log\left(
\frac{7}{2}md_m\,
q^{\frac{m+s+1}{2}}
\right)}
{
\log\left(
\frac{\sqrt q}{\frac{7}{2}md_m}
\right)
}.
\]
Consequently, for every integer $2m+G \leq n< p+G+ m-2$  satisfying the hypotheses of
Theorem \ref{teo: estimacion ptos racionales V} and the above inequality, system \eqref{eq:waring} admits an
$\mathbb{F}_q$--rational solution for every
$(b_1,\ldots,b_m)\in\mathbb{F}_q^m$. In particular,
\[
\gamma(d_1,\ldots,d_m,q)
\leq L:=\max \Bigg\{ 2m+G ,
\left\lfloor
\frac{
\log\left(
\frac{7}{2}md_m\,
q^{\frac{m+s+1}{2}}
\right)}
{
\log\left(
\frac{\sqrt q}{\frac{7}{2}md_m}
\right)
}
\right\rfloor+1\Bigg\}
\]
provided that
$L<p+m+G-2.$
In other words, $\gamma(d_1,\ldots,d_m,q)$ quantifies how many variables
are sufficient to ensure that every prescribed tuple in $\mathbb{F}_q^m$
can be simultaneously represented by power sums of degrees
$d_1,\ldots,d_m$.

\section{Computation of a rational point of a diagonal system}
 We consider the problem of computing  a rational solution of system \eqref{eq:waring}. For this purpose we can apply the algorithm described in \cite{CaMa06b} for computing an $\fq$--rational point of a \emph{general} absolutely irreducible affine complete intersection  $\fq$--variety $V\subset \A^n$ of positive dimension. This algorithm is based on the so--called geometric resolutions or Kronecker representations of varieties and requires that the input variety $V$ to be given by a reduced regular sequence $f_1 \klk f_m$  in $\fq[X_1 \klk X_n]$ that defines $V$. The complexity of this algorithm is dominated by the number $\mathcal{T}$ of arithmetic operations in $\fq$ required to evaluate the polynomials $f_1 \klk f_m$, and the maximum $\delta$  of the degrees of the intermediate varieties $V(f_1 \klk f_s)$, $1\le s \le m$. In these terms and omitting polylogarithmic factors, we have that the algorithm  computes an $\fq$--point of $V$  in time roughly $\mathcal{O}^{\sim}(\mathcal{T}\delta^2\log^2 q)$ with high probability of success.  Here $\mathcal{T}$ denotes the number of arithmetic operations in $\fq$ required to evaluate the polynomials $f_1 \klk f_m$, and $\delta$ is the maximum of the degrees of the varieties $V(f_1 \klk f_s)$, $1\le s \le m$ (see \cite[Theorem and Corollary 6.5]{CaMa06b} for details).
Thus, assume that we are given the polynomials $f_j=p_{d_j}^{(n)}-b_j, 1\leq j \leq m$,
defining system \eqref{eq:waring}. Further, assume that the hypotheses of Theorems \ref{teo: V es interseccion completa}, \ref{teo:props_clausura_infinito} and  \ref{teo: pcl V interseccion completa}   hold. Then, $f_1 \klk f_m$ is a reduced regular sequence in $\fq[X_1 \klk X_n]$ that defines an the  absolutely irreducible complete intersection affine variety $\mathcal{V}=V(f_1 \klk f_m)$ of degree $d_1\cdots d_m$. Assume that $q> 8n^2d_m(d_1\cdots d_m)^4$ as required in \cite[Corollary 6.5]{CaMa06b}. Note that our  polynomials $f_1 \klk f_m$ can be evaluated with $\mathcal{T}:= \mathcal{O}(n\log (d_1 \cdot \cdots \cdot d_m))$ arithmetic operations in $\fq$. Thus, by \cite[Corollary 6.5]{CaMa06b} we can compute an $\fq$-point of $\mathcal{V}$ in time roughly $\mathcal{O}^{\sim}((d_1\cdots d_m)^2\log^2 q)$ with probability of success at least $2/3>1/2$. Here, following the complexity convention of \cite{CaMa06b}, polynomial factors in \(n\) and \(d_m\) are omitted. Whether this complexity estimate can be sharpened for a structured system such as \eqref{eq:waring} is left for future work.

\end{document}